\documentclass{amsart}
\usepackage[utf8]{inputenc}
\usepackage{amsmath}
\usepackage{amsfonts}
\usepackage{amssymb}
\usepackage{amsthm}
\usepackage{mathtools}	
\usepackage{tikz}       
\usepackage{tikz-cd}	
\usepackage{hyperref}	
\usepackage{enumitem}
\usepackage{graphicx}

\newcommand{\F}{\ensuremath{\mathbb{F}} }
\newcommand{\N}{\ensuremath{\mathbb{N}} }
\newcommand{\Z}{\ensuremath{\mathbb{Z}} }

\newcommand{\R}{\ensuremath{\mathbb{R}} }

\newcommand{\gr}{\mathrm{gr}}
\newcommand{\CFK}{\mathit{CFK}}

\newtheorem{theorem}{Theorem}[section]

\newtheorem{proposition}[theorem]{Proposition}

\newtheorem{lemma}[theorem]{Lemma}
\newtheorem{corollary}[theorem]{Corollary}
\newtheorem{question}[theorem]{Question}
\theoremstyle{definition}
\newtheorem{definition}[theorem]{Definition}

\newtheorem{observation}[theorem]{Observation}

\theoremstyle{remark}
\newtheorem*{remark}{Remark}

\usepackage{subcaption}
\usepackage{tikz}
\usetikzlibrary{arrows}
\usetikzlibrary{decorations.pathreplacing,calligraphy}

\colorlet{dark green}{green!50!black}

\newcommand{\K}[1]{\reflectbox{\rotatebox[origin=c]{270}{$K$}}_#1}

\title{Knots of low knot Floer width}
\author{David Popovi\'c}
\address{Department of Mathematics\\
University of California, Los Angeles}
\email{dpopovic@math.ucla.edu}

\date{May 2024}

\begin{document}
\begin{abstract}
This paper classifies the chain homotopy equivalence types of knot Floer complexes $\CFK_{\F[U,V]}(K)$ of knot Floer width $2$. They have no nontrivial local systems. As an application, this shows that all Montesinos knots admit a basis that can be simultaneously horizontally and vertically simplified.
\end{abstract}
\maketitle

\section{Introduction}
Knot Floer homology is a powerful link invariant that was constructed in \cite{ozsvath2004holomorphicKnots} and \cite{rasmussen2003floer} based on an earlier family of $3$-manifold invariants called Heegaard Floer homology \cite{ozsvath2004holomorphic}. In its most general version, it is given by the chain homotopy type of a symmetric chain complex $\CFK_{\F[U,V]}(K)$ over $\F[U,V]$, equipped with a $\Z\oplus\Z$ grading $(\gr_U, \gr_V)$. By considering the grading $\delta = \frac{1}{2}(\gr_U+\gr_V)$ instead, one can define the \emph{knot Floer width} $w(K)$ of $K \subset S^3$ as $\max\delta-\min\delta+1$. The knot invariant $w(K)$ has recently been of interest, since it has been shown that it can be used to provide lower bounds on the Turaev genus of $K$ \cite{lowrance2008knot}, dealternating number of $K$ \cite{truong2023note}, and invariant $\beta(K)$ \cite{stipsicz2022note}.

\vspace{1em}

In this paper we study the $\Z\oplus\Z$ graded chain homotopy types associated to knots of low width. The case $w(K)=1$ is completely understood by the results in \cite{petkova2013cables} -- the knot Floer complexes split into a width $1$ staircase and a number of width $1$ squares. See also Section \ref{subsection:width 1} for precise restatement of these results. The classification of chain homotopy types associated to knots with $w(K)=2$ is significantly more challenging and the main result of this paper.
\begin{theorem}\label{thm:width2}
Let $K \subset S^3$ be a knot with $w(K)=2$ and let $\CFK_{\F[U,V]}(K)$ be its knot Floer complex. Then the chain homotopy type of $\CFK_{\F[U,V]}(K)$ splits uniquely as a direct sum of a width $2$ standard complex (Def. \ref{def:standard complex}, see also Fig. \ref{fig:lifts of width 2 std cxs}) and some trivial local systems of the following shapes:
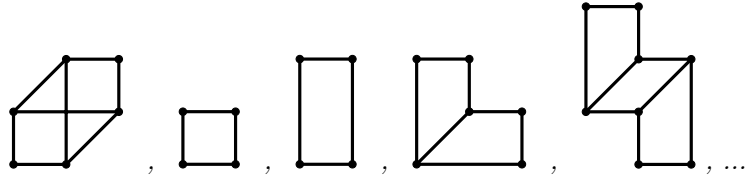
\begin{center}
\begin{tikzpicture}[scale = 0.7]
    \draw[black, very thick] (1,0) -- (1,1) -- (3,1) -- (3,2) -- (2,2) -- (2,0) -- (1,0);
    \draw[black, very thick] (2,2) -- (1,1);
    \draw[black, very thick] (3,1) -- (2,0);
    \filldraw[black] (1,1) circle (2pt) node[anchor=south east]{};
    \filldraw[black] (3,1) circle (2pt) node[anchor=south west]{};
    \filldraw[black] (3,2) circle (2pt) node[anchor=south west]{};
    \filldraw[black] (2,2) circle (2pt) node[anchor=south east]{};
    \filldraw[black] (2,0) circle (2pt) node[anchor=south east]{};
    \filldraw[black] (1,0) circle (2pt) node[anchor=south east]{};
\end{tikzpicture}
,
\begin{tikzpicture}[scale = 0.7]
    \draw[black, very thick] (1,1) -- (2,1) -- (2,2) -- (1,2) -- (1,1);
    \filldraw[black] (1,1) circle (2pt) node[anchor=south east]{};
    \filldraw[black] (2,1) circle (2pt) node[anchor=south west]{};
    \filldraw[black] (2,2) circle (2pt) node[anchor=south west]{};
    \filldraw[black] (1,2) circle (2pt) node[anchor=south east]{};
\end{tikzpicture}
,
\begin{tikzpicture}[scale = 0.7]
    \draw[black, very thick] (1,1) -- (2,1) -- (2,3) -- (1,3) -- (1,1);
    \filldraw[black] (1,1) circle (2pt) node[anchor=south east]{};
    \filldraw[black] (2,1) circle (2pt) node[anchor=south west]{};
    \filldraw[black] (2,3) circle (2pt) node[anchor=south west]{};
    \filldraw[black] (1,3) circle (2pt) node[anchor=south east]{};
\end{tikzpicture}
,
\begin{tikzpicture}[scale = 0.7]
    \draw[black, very thick] (1,1) -- (3,1) -- (3,2) -- (2,2) -- (2,3) -- (1,3) -- (1,1);
    \draw[black, very thick] (2,2) -- (1,1);
    \filldraw[black] (1,1) circle (2pt) node[anchor=south east]{};
    \filldraw[black] (3,1) circle (2pt) node[anchor=south west]{};
    \filldraw[black] (3,2) circle (2pt) node[anchor=south west]{};
    \filldraw[black] (2,3) circle (2pt) node[anchor=south east]{};
    \filldraw[black] (2,2) circle (2pt) node[anchor=south east]{};
    \filldraw[black] (1,3) circle (2pt) node[anchor=south east]{};
\end{tikzpicture}
,
\begin{tikzpicture}[scale = 0.7]
    \draw[black, very thick] (0,0)--(1,0)--(1,-1)--(2,-1)--(2,-3)--(1,-3)--(1,-2)--(0,-2)--(0,0);
    \draw[black, very thick] (0,-2)--(1,-1);
    \draw[black, very thick] (1,-2)--(2,-1);
    \filldraw[black] (0,0) circle (2pt) node[anchor=south east]{};
    \filldraw[black] (1,0) circle (2pt) node[anchor=south east]{};
    \filldraw[black] (1,-1) circle (2pt) node[anchor=south east]{};
    \filldraw[black] (2,-1) circle (2pt) node[anchor=south east]{};
    \filldraw[black] (2,-3) circle (2pt) node[anchor=south east]{};
    \filldraw[black] (1,-3) circle (2pt) node[anchor=south east]{};
    \filldraw[black] (1,-2) circle (2pt) node[anchor=south east]{};
    \filldraw[black] (0,-2) circle (2pt) node[anchor=south east]{};
\end{tikzpicture}
, ...
\end{center}
and their reflections.
\end{theorem}
The primary motivation for studying this question is three-fold. A complete classification of knot Floer complexes, without any restrictions on $w(K)$, exists over the quotient ring $\mathcal{R} = \frac{\F[U,V]}{(UV)}$ by the author's earlier work \cite[Theorem 1.2]{popovic2023link}. It establishes a chain homotopy equivalence 
$$\CFK_{\mathcal{R}}(K) \simeq C(a_1, \dots, a_{2n}) \oplus L_1 \oplus \dots \oplus L_k$$
where $C(a_1, \dots, a_{2n})$ is a standard complex, $L_1, \dots, L_k$ are local systems (Def. \ref{def:local system}), and all summands are unique up to permutation and chain homotopy equivalence. Refining this classification to one over $\F[U,V]$ seems intractable in general, but is just about possible for knots of width $1$ and $2$.

\vspace{1em}

Secondly, in none of the examples that have been computed thus far, a non-trivial local system has been observed. Theorem \ref{thm:width2} provides a partial explanation as for why this is the case -- in the realm of low crossing numbers where $\CFK_{\F[U,V]}(K)$ can be systematically computed, the majority of the knots have width $1$ or $2$ and thus their local systems will always be trivial.

\vspace{1em}

Finally, and relatedly, the following question goes back at least as far as \cite{hom2015infinite}, and probably further.
\begin{question}
Does every knot Floer complex $\CFK_{\F[U,V]}(K)$ admit a basis that is simultaneously vertically and horizontally simplified?
\end{question}
Note that in terms of the classification theorem over $\mathcal{R}$, this is equivalent to whether the local systems $L_i$ are always trivial. Theorem \ref{thm:width2} shows that this is the case for all knots of width $2$.
\begin{corollary}
All knots of knot Floer width $2$ admit a basis that is simultaneously horizontally and vertically simplified.
\end{corollary}
It was proven in \cite[Corollary 1.3]{stipsicz2022note} that Montesinos knots have width $\leq 2$. This immediately lets us state the following topological corollary.
\begin{corollary}\label{cor:montesinos}
Montesinos knots admit a basis that is simultaneously horizontally and vertically simplified.
\end{corollary}
We note that pretzel knots are a special family of Montesinos knots and that not even their knot Floer complexes are known in general. Theorem \ref{thm:width2} and Corollary \ref{cor:montesinos} significantly restrict their chain homotopy type.

\vspace{1em}

It would perhaps also be interesting to note that all but one of the local systems from Theorem \ref{thm:width2} \emph{do} appear as summands of $\CFK_{\F[U,V]}(K)$ for various knots $K$. With the exception of one special shape, we exhibit them in certain iterated cables of the figure eight knot in Section \ref{sec:realization}.

\subsection*{Acknowledgement}
I would like to thank Sucharit Sarkar for his advice, reading the draft of this paper, and many helpful discussions. I would also like to thank Keegan Boyle, Wenzhao Chen, Mihai Marian, Ko Honda, and Liam Watson for interesting conversations. This work was partially supported by NSF Grant DMS-2003483.

\section{Background}
We briefly define the terms we will be operating with in this paper. Let $K \subset S^3$ be a knot and $\F=\Z/2\Z$. The knot Floer complex $\CFK_{\F[U,V]}(K)$ is a finitely generated $\Z\oplus\Z$ graded chain complex over $\F[U,V]$ with a differential $\partial$ and grading $\gr=(\gr_U, \gr_V)$ satisfying $\gr(U)=(-2,0)$, $\gr(V)=(0,-2)$, and $\gr(\partial)=(-1,-1)$. See \cite{ozsvath2004holomorphicKnots} for the original paper where knot Floer complexes were defined and \cite{zemke2019connected} for a longer exposition of the construction that matches our conventions. One can collapse $\gr_U$ and $\gr_V$ to a single grading $\delta = \frac{1}{2}(\gr_U + \gr_V)$.
\begin{definition}
The knot Floer \emph{width} of a knot $K$ is 
$$w(K) = \max\{ \delta(x) \ | \ x \in \CFK_{\F[U,V]}(K) \} - \min\{ \delta(x) \ | \ x \in \CFK_{\F[U,V]}(K) \} + 1.$$
\end{definition}
Knot Floer width is traditionally defined in terms of the Maslov grading $M$ and the Alexander grading $A$. It is the number of different diagonals needed to depict the complex $\CFK_{\F[U,V]}(K)$ in the plane whose coordinate axes are $A$ and $M$. Since $M = \gr_U$ and $A=\frac{1}{2}(\gr_U-\gr_V)$, each diagonal is specified by the quantity $M-A = \delta$ and the two definitions coincide.

\vspace{1em}

The most important tool we will use in this paper is the classification theorem over $\mathcal{R} = \F[U,V]/(UV)$. In order to state it, we first need to define standard complexes and local systems. Both of these concepts are more easily conceptualized pictorially (see Figure \ref{fig:cfk splitting}) rather than via formal definitions, but we include the latter nonetheless. For a longer discussion of standard complexes and local systems we recommend \cite[Section 4]{dai2021more} and \cite[Section 3]{popovic2023link} where these terms were originally defined.

\begin{definition}\label{def:standard complex}
Let $n \in \N_0$ and let $a_1, \dots, a_{2n}$ be a sequence of nonzero integers. The \emph{standard complex} $C(a_1, \dots, a_{2n})$ is a free chain complex over $\F[U,V]/(UV)$ with a distinguished basis $B=\{x_{0}, \dots, x_{2n}\}$ and a differential $\partial$ defined as follows. For each odd $i$, there is a horizontal arrow of length $|a_i|$ connecting $x_i$ and $x_{i-1}$. For each even $i$, there is a vertical arrow of length $|a_i|$ connecting $x_i$ and $x_{i-1}$. The direction of the arrow is determined by the sign of $a_i$. If $a_i > 0$, then the arrow goes from $x_i$ to $x_{i-1}$, and if $a_i < 0$, then the arrow goes from $x_{i-1}$ to $x_i$. The $\Z\oplus\Z$ grading on $C(a_1, \dots, a_n)$ is uniquely determined by the condition $\gr_U(x_0)=0$.
\end{definition}
\begin{definition}\label{def:local system}
Let $(L, \partial)$ be a finitely generated free chain complex over $\mathcal{R}$ with no arrows of length $0$. Then $L$ is an \emph{indecomposable local system} if it
    \begin{enumerate}
        \item admits a simplified decomposition (\cite[Definition 3.1]{popovic2023link}),
        \item is indecomposable as a chain complex over $\mathcal{R}$, and
        \item has torsion homology (\cite[Definition 2.2]{popovic2023link}).
    \end{enumerate}
A \emph{local system} is a direct sum of indecomposable local systems of the same shape (\cite[Definition 3.5]{popovic2023link}) and in the same position in the plane.
\end{definition}
We are now in a position to state the classification theorem over $\mathcal{R} = \F[U,V]/(UV)$. See Figure \ref{fig:cfk splitting} for the pictorial restatement of the theorem.
\begin{theorem}[\cite{popovic2023link}]\label{thm:classification with UV=0}
Let $K \subset S^3$ be a knot and let $\CFK_{\mathcal{R}}(K)$ be its link Floer complex. Then
\[\CFK_{\mathcal{R}}(K) \simeq C(a_1, \dots, a_{2n}) \oplus L_1 \oplus \dots \oplus L_k\]
where $C(a_1, \dots, a_{2n})$ is a standard complex and $L_1, \dots, L_k$ are local systems. Moreover, the direct summands are unique up to permutation.
\end{theorem}
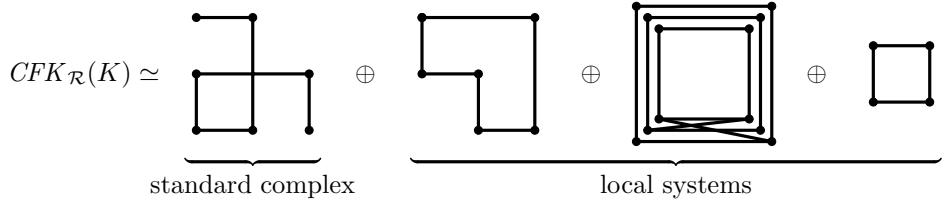
\begin{figure}[t]
    \begin{tikzpicture}[scale=0.75]
        \def\a{0.2}
        \def\b{0.5}

        \draw[] (-1,0) node[]{$\CFK_{\mathcal{R}}(K) \simeq$};
        \draw[very thick] (1,1)--(2,1)--(2,-1)--(1,-1)--(1,0)--(3,0)--(3,-1) {};
        \draw[] (4,0) node[]{$\oplus$};
        \draw[very thick] (5,1)--(7,1)--(7,-1)--(6,-1)--(6,0)--(5,0)--(5,1) {};
        \draw[] (8,0) node[]{$\oplus$};
        \draw[very thick] (9,1)--(11,1)--(11,-1)--(9,-1)--(9,1) {};
        \draw[very thick] (9-\a,1+\a)--(11+\a,1+\a)--(11+\a,-1-\a)--(9-\a,-1-\a)--(9-\a,1+\a) {};
        \draw[very thick] (9+\a,1-\a)--(11-\a,1-\a)--(11-\a,-1+\a)--(9+\a,-1+\a)--(9+\a,1-\a) {};
        \draw[very thick] (11-\a,-1+\a)--(9,-1) {};
        \draw[very thick] (11+\a,-1-\a)--(9+\a,-1+\a) {};
        \draw[] (12,0) node[]{$\oplus$};
        \draw[very thick] (13,0.5)--(14,0.5)--(14,-0.5)--(13,-0.5)--(13,0.5) {};

        \filldraw[] (1,1) circle (2pt) node[anchor=north west]{};
        \filldraw[] (2,1) circle (2pt) node[anchor=north west]{};
        \filldraw[] (2,-1) circle (2pt) node[anchor=north west]{};
        \filldraw[] (1,-1) circle (2pt) node[anchor=north west]{};
        \filldraw[] (1,0) circle (2pt) node[anchor=north west]{};
        \filldraw[] (3,0) circle (2pt) node[anchor=north west]{};
        \filldraw[] (3,-1) circle (2pt) node[anchor=north west]{};
        \filldraw[] (5,1) circle (2pt) node[anchor=north west]{};
        \filldraw[] (7,1) circle (2pt) node[anchor=north west]{};
        \filldraw[] (7,-1) circle (2pt) node[anchor=north west]{};
        \filldraw[] (6,-1) circle (2pt) node[anchor=north west]{};
        \filldraw[] (6,0) circle (2pt) node[anchor=north west]{};
        \filldraw[] (5,0) circle (2pt) node[anchor=north west]{};
        \filldraw[] (5,1) circle (2pt) node[anchor=north west]{};
        \filldraw[] (9,1) circle (2pt) node[anchor=north west]{};
        \filldraw[] (11,1) circle (2pt) node[anchor=north west]{};
        \filldraw[] (11,-1) circle (2pt) node[anchor=north west]{};
        \filldraw[] (9,-1) circle (2pt) node[anchor=north west]{};
        \filldraw[] (9+\a,1-\a) circle (2pt) node[anchor=north west]{};
        \filldraw[] (11-\a,1-\a) circle (2pt) node[anchor=north west]{};
        \filldraw[] (11-\a,-1+\a) circle (2pt) node[anchor=north west]{};
        \filldraw[] (9+\a,-1+\a) circle (2pt) node[anchor=north west]{};
        \filldraw[] (9-\a,1+\a) circle (2pt) node[anchor=north west]{};
        \filldraw[] (11+\a,1+\a) circle (2pt) node[anchor=north west]{};
        \filldraw[] (11+\a,-1-\a) circle (2pt) node[anchor=north west]{};
        \filldraw[] (9-\a,-1-\a) circle (2pt) node[anchor=north west]{};
        \filldraw[] (13,0.5) circle (2pt) node[anchor=north west]{};
        \filldraw[] (14,0.5) circle (2pt) node[anchor=north west]{};
        \filldraw[] (14,-0.5) circle (2pt) node[anchor=north west]{};
        \filldraw[] (13,-0.5) circle (2pt) node[anchor=north west]{};

        \draw [very thick, decorate, decoration = {calligraphic brace, mirror}] (1-\a,-1-\b)--(3+\a,-1-\b);
        \draw [very thick, decorate, decoration = {calligraphic brace, mirror}] (5-\a,-1-\b)--(14+\a,-1-\b);

        \draw[] (2,-2) node[]{standard complex};
        \draw[] (9.5,-2) node[]{local systems};
        
    \end{tikzpicture}
    \caption{Schematic depiction of Theorem \ref{thm:classification with UV=0}. The knot Floer complex of a knot $K$ splits into a standard complex and local systems.}
    \label{fig:cfk splitting}
\end{figure}
This theorem tells us what knot Floer complexes look like over $\mathcal{R}$. In the rest of the paper, we will use the techniques of \cite{popovic2023algebraic} to characterize which of those of width $2$ can be lifted to complexes over $\F[U,V]$. 

\section{Preliminary observations}
In this section we make some preliminary observations about knot Floer complexes of low width.
\begin{observation}\label{observation: 1}
    Let $K$ be a knot. If $w(K)=1$, then $\CFK_{\F[U,V]}(K)$ consists only of horizontal and vertical arrows of length $1$. If $w(K)=2$, then $\CFK_{\F[U,V]}(K)$ consists of horizontal and vertical arrows of length at most $2$, and diagonal arrows that move by $1$ in both directions. Moreover, the length $1$ horizontal and vertical arrows preserve $\delta$, and the other arrows increase it by $1$.
\end{observation}
\begin{definition}
    Let $(C, \partial)$ be a free chain complex over $\F[U,V]$. For generators $x, y \in C$ and $a, b \in \N_0$, let $\langle \partial x, U^aV^by\rangle \in \F$ denote the coefficient of $U^aV^by$ in $\partial x$.
\end{definition}
\begin{proof}[Proof of Observation \ref{observation: 1}]
Let $x, y \in \CFK_{\F[U,V]}(K)$ be generators with $\langle \partial x, U^aV^b y \rangle =1$. Then $\gr(U^aV^by)=\gr(x)-(1,1)$ and so $\gr(y)=\gr(x)+(2a-1,2b-1)$. Hence $\delta(y) = \delta(x) + a+b-1$. The conclusion follows by the upper bounds on $w(K)$.
\end{proof}
\begin{observation}\label{obs:enough to do single local systems}
Let $K$ be a knot with $w(K) \leq 2$ and $L$ a direct summand of $\CFK_{\mathcal{R}}(K)$. Then $L$ is itself a mod $UV$ reduction of a chain complex over $\F[U,V]$.
\end{observation}
\begin{proof}
If $L$ itself is not a chain complex over $\F[U,V]$, then $\partial^2 x \neq 0$ for some $x\in L$, say $\langle \partial^2 x, U^aV^b y \rangle =1$ for some $a, b \in \{1,2\}$. Since $w(K) \leq 2$ and by our previous observation, $x$ and $y$ must be away by $1$ step in one direction and at most $2$ steps in another direction. This means that the alternative way of traveling from $x$ to $y$ in $\CFK_{\F[U,V]}(K)$ in $2$ steps must contain at least one non-diagonal arrow. Therefore, the alternative way also passes through a generator in $L$, \emph{i.e.} $\partial^2x=0$ over $\F[U,V]$ can be restored by adding a diagonal arrow to $L$.
\end{proof}
With this in mind, it is sufficient to classify the standard complexes and local systems that can be lifted to chain complexes over $\F[U,V]$ by themselves. Our previous work contains an elementary algorithm \cite[Algorithm 3.12]{popovic2023algebraic} that characterizes standard complexes that can be lifted to chain complexes over $\frac{\F[U,V]}{(U^2V^2)}$. In the presence of $w(K) \leq 2$ restriction, this is the same as lifting to $\F[U,V]$. See Figure \ref{fig:lifts of width 2 std cxs} for some such lifts of width $2$ standard complexes. Therefore, we focus on the local systems in the next section.
\begin{figure}
\centering
\begin{subfigure}{0.30\textwidth}
    $$
        \begin{tikzpicture}[scale = 0.8]
            \draw[step=1.0,gray,thin] (-0.5, -2.5) grid (2.5,0.5);
            \draw[black, very thick] (0,0)--(1,0)--(1,-1)--(2,-1)--(2,-2) -- (0, -2) -- (0, -1);
            \draw[black, very thick] (0,-1)--(1,0);
            \draw[black, very thick] (0,-2)--(1,-1);
            \filldraw[black] (0,0) circle (2pt) node[anchor=south east]{};
            \filldraw[black] (1,0) circle (2pt) node[anchor=south east]{};
            \filldraw[black] (1,-1) circle (2pt) node[anchor=south east]{};
            \filldraw[black] (2,-1) circle (2pt) node[anchor=south east]{};
            \filldraw[black] (2,-2) circle (2pt) node[anchor=south east]{};
            \filldraw[black] (0,-2) circle (2pt) node[anchor=south east]{};
            \filldraw[black] (0,-1) circle (2pt) node[anchor=south east]{};

            \draw[] (1,-3) node[]{\footnotesize $C(1, -1, 1, -1, -2, 1)$};
        \end{tikzpicture}
    $$
    \caption{}
    \label{fig:width 2 ex a}
\end{subfigure}
\begin{subfigure}{0.30\textwidth}
    $$
        \begin{tikzpicture}[scale = 0.8]
            \def\a{0.1}
            \draw[step=1.0,gray,thin] (-0.5, -2.5) grid (2.5,1.5);
            \draw[black, very thick] (2,0)--(1,-1+\a)--(0,-1)--(0,0)--(2,0)--(2,-2)--(1,-2)--(1,-1-\a)--(2,0);
            \draw[black, very thick] (1,-1+\a)--(1,1) -- (0,1);
            \draw[black, very thick] (0,0)--(1,1);
            \filldraw[black] (0,0) circle (2pt) node[anchor=south east]{};
            \filldraw[black] (1,1) circle (2pt) node[anchor=south east]{};
            \filldraw[black] (0,1) circle (2pt) node[anchor=south east]{};
            \filldraw[black] (0,-1) circle (2pt) node[anchor=south east]{};
            \filldraw[black] (1,-1+\a) circle (2pt) node[anchor=south east]{};
            \filldraw[black] (1,-1-\a) circle (2pt) node[anchor=south east]{};
            \filldraw[black] (1,-2) circle (2pt) node[anchor=south east]{};
            \filldraw[black] (2,-2) circle (2pt) node[anchor=south east]{};
            \filldraw[black] (2,0) circle (2pt) node[anchor=south east]{};

            \draw[] (1,-3) node[]{\footnotesize $C(1, -2, -1, 1, 2, -2, -1, 1)$};
        \end{tikzpicture}
    $$
    \caption{}
    \label{fig:width 2 ex b}
\end{subfigure}
\begin{subfigure}{0.30\textwidth}
    $$
        \begin{tikzpicture}[scale = 0.8]
            \draw[step=1.0,gray,thin] (-1.5, -1.5) grid (2.5,2.5);
            \def\a{0.12}
            \draw[black, very thick] (0-\a,0-\a)--(-1-\a,0-\a)--(-1-\a,2)--(0,2)--(0,1+\a)--(1+\a,1+\a)--(1+\a,0)--(2,0)--(2,-1)--(0+\a,-1)--(0+\a,0+\a)--(-1,0+\a)--(-1,1)--(1-\a,1)--(1-\a,0);
            \draw[black, very thick] (-1-\a,0-\a)--(0,1+\a)--(-1,0+\a);
            \draw[black, very thick] (0,2)--(-1,1);
            \draw[black, very thick] (1+\a,1+\a) -- (0+\a,0+\a)--(1-\a,1);
            \draw[black, very thick] (1+\a,0)--(0+\a,-1)--(1-\a,0);
            \draw[black, very thick] (0-\a, 0-\a) to[out = 20, in=240] (1+\a, 1+\a);
            \filldraw[black] (0-\a,0-\a) circle (2pt) node[anchor=south east]{};
            \filldraw[black] (-1-\a,0-\a) circle (2pt) node[anchor=south east]{};
            \filldraw[black] (-1-\a,2) circle (2pt) node[anchor=south east]{};
            \filldraw[black] (0,2) circle (2pt) node[anchor=south east]{};
            \filldraw[black] (0,1+\a) circle (2pt) node[anchor=south east]{};
            \filldraw[black] (1+\a,1+\a) circle (2pt) node[anchor=south east]{};
            \filldraw[black] (1+\a,0) circle (2pt) node[anchor=south east]{};
            \filldraw[black] (2,0) circle (2pt) node[anchor=south east]{};
            \filldraw[black] (2,-1) circle (2pt) node[anchor=south east]{};
            \filldraw[black] (0+\a,-1) circle (2pt) node[anchor=south east]{};
            \filldraw[black] (0+\a,0+\a) circle (2pt) node[anchor=south east]{};
            \filldraw[black] (-1,0+\a) circle (2pt) node[anchor=south east]{};
            \filldraw[black] (-1,1) circle (2pt) node[anchor=south east]{};
            \filldraw[black] (1-\a,1) circle (2pt) node[anchor=south east]{};
            \filldraw[black] (1-\a,0) circle (2pt) node[anchor=south east]{};

            \draw[] (1,-2) node[]{\footnotesize \textcolor{white}{$C$}};
        \end{tikzpicture}
    $$
    \caption{}
    \label{fig:width 2 ex c}
\end{subfigure}
\caption{Lifts of some width $2$ standard complexes to chain complexes over $\F[U,V]$. Figure (\textsc{c}) depicts the lift of the standard complex $C(-1, 2, 1, -1, 1, -1, 1, -1, -2, 1, -1, 1, 2, -1)$.}
\label{fig:lifts of width 2 std cxs}
\end{figure}
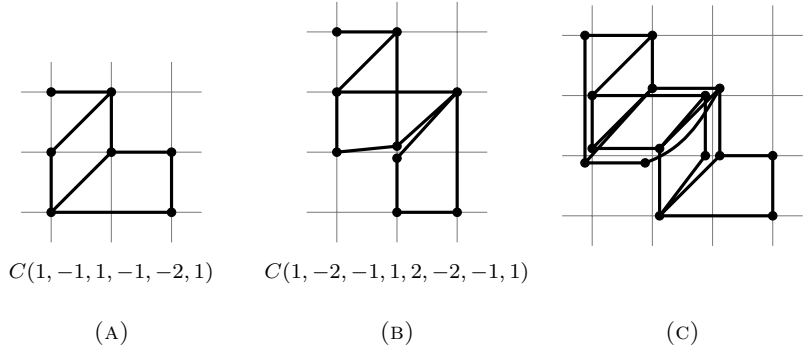

\section{Knots of low width}
\subsection{Names of chain  complexes}
In this subsection we introduce certain chain complexes that will play a role in the subsequent classification. 
\begin{definition}
A \emph{width $1$ staircase} is a standard complex $C(1, -1, \dots, 1, -1)$, a standard complex $C(-1, 1, \dots, -1, 1)$, or one of these with an additional arrow of length $1$ at one of the endpoints.
\end{definition}
There are four different types of width $1$ staircases depending on whether one starts and ends with horizontal or vertical arrows. They are depicted in Figure \ref{fig:width 1 staircases}.
\begin{figure}
\centering
\begin{subfigure}{0.25\textwidth}
    $$
        \begin{tikzpicture}[scale = 0.8]
            \draw[step=1.0,gray,thin] (-0.5, -2.5) grid (2.5,0.5);
            \draw[black, very thick] (0,0)--(1,0)--(1,-1)--(2,-1)--(2,-2);
            \filldraw[black] (0,0) circle (2pt) node[anchor=south east]{};
            \filldraw[black] (1,0) circle (2pt) node[anchor=south east]{};
            \filldraw[black] (1,-1) circle (2pt) node[anchor=south east]{};
            \filldraw[black] (2,-1) circle (2pt) node[anchor=south east]{};
            \filldraw[black] (2,-2) circle (2pt) node[anchor=south east]{};
        \end{tikzpicture}
    $$
    \caption{}
    \label{fig:width 1 staircase a}
\end{subfigure}
\begin{subfigure}{0.25\textwidth}
    $$
        \begin{tikzpicture}[scale = 0.8]
            \draw[step=1.0,gray,thin] (-0.5, -2.5) grid (2.5,0.5);
            \draw[black, very thick] (0,0)--(0,-1)--(1,-1)--(1,-2)--(2,-2);
            \filldraw[black] (0,0) circle (2pt) node[anchor=south east]{};
            \filldraw[black] (0,-1) circle (2pt) node[anchor=south east]{};
            \filldraw[black] (1,-1) circle (2pt) node[anchor=south east]{};
            \filldraw[black] (1,-2) circle (2pt) node[anchor=south east]{};
            \filldraw[black] (2,-2) circle (2pt) node[anchor=south east]{};
        \end{tikzpicture}
    $$
    \caption{}
    \label{fig:width 1 staircase b}
\end{subfigure}
\begin{subfigure}{0.25\textwidth}
    $$
        \begin{tikzpicture}[scale = 0.8]
            \draw[step=1.0,gray,thin] (-0.5, -2.5) grid (2.5,0.5);
            \draw[black, very thick] (1,0)--(1,-1)--(2,-1)--(2,-2);
            \filldraw[black] (1,0) circle (2pt) node[anchor=south east]{};
            \filldraw[black] (1,-1) circle (2pt) node[anchor=south east]{};
            \filldraw[black] (2,-1) circle (2pt) node[anchor=south east]{};
            \filldraw[black] (2,-2) circle (2pt) node[anchor=south east]{};
        \end{tikzpicture}
    $$
    \caption{}
    \label{fig:width 1 staircase c}
\end{subfigure}
\begin{subfigure}{0.25\textwidth}
    $$
        \begin{tikzpicture}[scale = 0.8]
            \draw[step=1.0,gray,thin] (-0.5, -2.5) grid (2.5,0.5);
            \draw[black, very thick] (0,0)--(1,0)--(1,-1)--(2,-1);
            \filldraw[black] (0,0) circle (2pt) node[anchor=south east]{};
            \filldraw[black] (1,0) circle (2pt) node[anchor=south east]{};
            \filldraw[black] (1,-1) circle (2pt) node[anchor=south east]{};
            \filldraw[black] (2,-1) circle (2pt) node[anchor=south east]{};
        \end{tikzpicture}
    $$
    \caption{}
    \label{fig:width 1 staircase d}
\end{subfigure}
\caption{Different types of width $1$ staircases.}
\label{fig:width 1 staircases}
\end{figure}
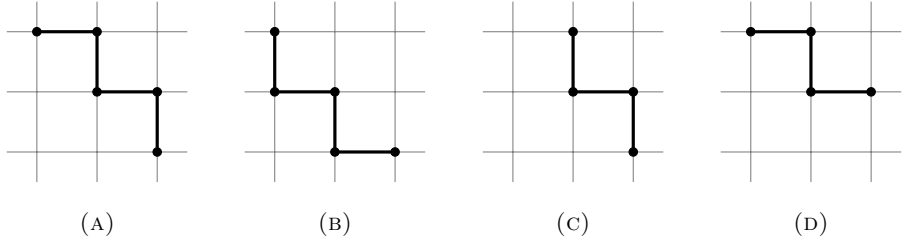
\begin{definition}\label{def:width 2 staircase}
A \emph{width $2$ staircase} is a trivial local system consisting of two width $1$ staircases connected by two length $2$ arrows and some diagonal arrows.
\end{definition}
There are three different types of width $2$ staircases and they are depicted in Figure \ref{fig:width 2 staircases}.
\begin{figure}
\centering
\begin{subfigure}{0.30\textwidth}
    $$
        \begin{tikzpicture}[scale = 0.8]
            \draw[step=1.0,gray,thin] (-0.5, -3.5) grid (3.5,0.5);
            \draw[black, very thick] (0,0)--(1,0)--(1,-1)--(2,-1)--(2,-2)--(3,-2)--(3,-3)--(1,-3)--(1,-2)--(0,-2)--(0,0);
            \draw[black, very thick] (0,-2)--(1,-1);
            \draw[black, very thick] (1,-2)--(2,-1);
            \draw[black, very thick] (1,-3)--(2,-2);
            \filldraw[black] (0,0) circle (2pt) node[anchor=south east]{};
            \filldraw[black] (1,0) circle (2pt) node[anchor=south east]{};
            \filldraw[black] (1,-1) circle (2pt) node[anchor=south east]{};
            \filldraw[black] (2,-1) circle (2pt) node[anchor=south east]{};
            \filldraw[black] (2,-2) circle (2pt) node[anchor=south east]{};
            \filldraw[black] (3,-2) circle (2pt) node[anchor=south east]{};
            \filldraw[black] (3,-3) circle (2pt) node[anchor=south east]{};
            \filldraw[black] (1,-3) circle (2pt) node[anchor=south east]{};
            \filldraw[black] (1,-2) circle (2pt) node[anchor=south east]{};
            \filldraw[black] (0,-2) circle (2pt) node[anchor=south east]{};
        \end{tikzpicture}
    $$
    \caption{}
    \label{fig:width 2 staircase a}
\end{subfigure}
\begin{subfigure}{0.30\textwidth}
    $$
        \begin{tikzpicture}[scale = 0.8]
            \draw[step=1.0,gray,thin] (-0.5, -3.5) grid (2.5,0.5);
            \draw[black, very thick] (0,0)--(1,0)--(1,-1)--(2,-1)--(2,-3)--(1,-3)--(1,-2)--(0,-2)--(0,0);
            \draw[black, very thick] (0,-2)--(1,-1);
            \draw[black, very thick] (1,-2)--(2,-1);
            \filldraw[black] (0,0) circle (2pt) node[anchor=south east]{};
            \filldraw[black] (1,0) circle (2pt) node[anchor=south east]{};
            \filldraw[black] (1,-1) circle (2pt) node[anchor=south east]{};
            \filldraw[black] (2,-1) circle (2pt) node[anchor=south east]{};
            \filldraw[black] (2,-3) circle (2pt) node[anchor=south east]{};
            \filldraw[black] (1,-3) circle (2pt) node[anchor=south east]{};
            \filldraw[black] (1,-2) circle (2pt) node[anchor=south east]{};
            \filldraw[black] (0,-2) circle (2pt) node[anchor=south east]{};
        \end{tikzpicture}
    $$
    \caption{}
    \label{fig:width 2 staircase b}
\end{subfigure}
\begin{subfigure}{0.30\textwidth}
    $$
        \begin{tikzpicture}[scale = 0.8]
            \draw[step=1.0,gray,thin] (-0.5, -3.5) grid (3.5,0.5);
            \draw[black, very thick] (0,-2)--(0,-1)--(1,-1)--(2,-1)--(2,-2)--(3,-2)--(3,-3)--(1,-3)--(1,-2)--(0,-2);
            \draw[black, very thick] (1,-2)--(2,-1);
            \draw[black, very thick] (1,-3)--(2,-2);
            \filldraw[black] (0,-1) circle (2pt) node[anchor=south east]{};
            \filldraw[black] (2,-1) circle (2pt) node[anchor=south east]{};
            \filldraw[black] (2,-2) circle (2pt) node[anchor=south east]{};
            \filldraw[black] (3,-2) circle (2pt) node[anchor=south east]{};
            \filldraw[black] (3,-3) circle (2pt) node[anchor=south east]{};
            \filldraw[black] (1,-3) circle (2pt) node[anchor=south east]{};
            \filldraw[black] (1,-2) circle (2pt) node[anchor=south east]{};
            \filldraw[black] (0,-2) circle (2pt) node[anchor=south east]{};
        \end{tikzpicture}
    $$
    \caption{}
    \label{fig:width 2 staircase c}
\end{subfigure}
\caption{Different types of width $2$ staircases.}
\label{fig:width 2 staircases}
\end{figure}
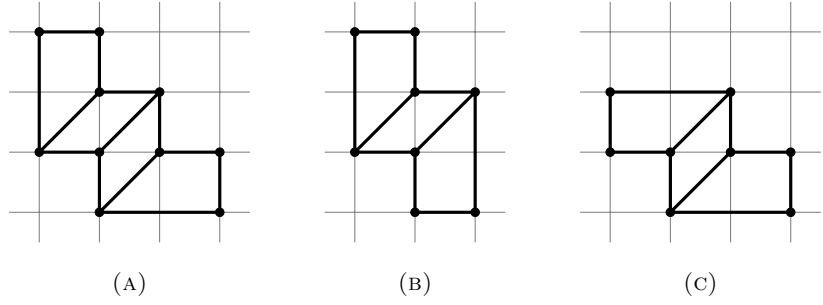

\begin{definition}
Let $(C, \partial)$ be a finitely generated chain complex over $\F[U,V]$.
\begin{enumerate}
    \item A \emph{loop} is a subcomplex $D$ of $C$ generated by 5 generators $a$, $b$, $c$, $d$, and $e$ such that $\partial$ restricted to $D$ satisfies $\partial a = 0$, $\partial b = U^2a + UVe$, $\partial c = Ud + Vb$, $\partial d = V^2e + UVa$ and $\partial e = 0$. We say that such a loop starts at $a$ and ends at $e$. A loop is drawn in Figure \ref{fig:loop} and its mirror in Figure \ref{fig:loop mirror}.
    \item A \emph{special shape} is a subcomplex $E$ of $C$ generated by 6 generators $a$, $b$, $c$, $d$, $e$, and $f$ such that $\partial$ restricted to $E$ satisfies $\partial a = Vf$, $\partial b = U^2a + UVe$, $\partial c = Ud + Vb$, $\partial d = V^2e + UVa$ and $\partial e = Uf$. A special shape is shown in Figure \ref{fig:special shape}.
\end{enumerate} 
\end{definition}
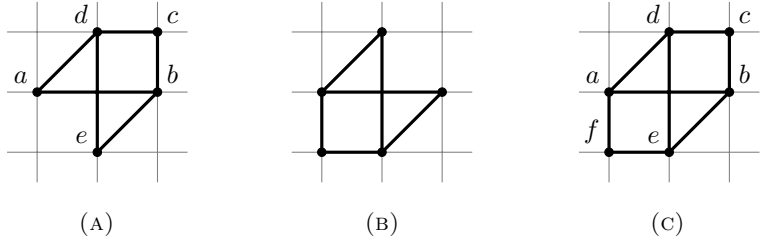
\begin{figure}
    \centering
    \begin{subfigure}{0.3\textwidth}
    $$
        \begin{tikzpicture}[scale = 0.8]
            \draw[step=1.0,gray,thin] (0.5,-0.5) grid (3.5,2.5);
            \draw[black, very thick] (1,1) -- (3,1) -- (3,2) -- (2,2) -- (2,0);
            \draw[black, very thick] (2,2) -- (1,1);
            \draw[black, very thick] (3,1) -- (2,0);
            \filldraw[black] (1,1) circle (2pt) node[anchor=south east]{$a$};
            \filldraw[black] (3,1) circle (2pt) node[anchor=south west]{$b$};
            \filldraw[black] (3,2) circle (2pt) node[anchor=south west]{$c$};
            \filldraw[black] (2,2) circle (2pt) node[anchor=south east]{$d$};
            \filldraw[black] (2,0) circle (2pt) node[anchor=south east]{$e$};
        \end{tikzpicture}
    $$
    \caption{}
    \label{fig:loop}
    \end{subfigure}
    \begin{subfigure}{0.3\textwidth}
    $$
        \begin{tikzpicture}[scale = 0.8]
            \draw[step=1.0,gray,thin] (0.5,-0.5) grid (3.5,2.5);
            \draw[black, very thick] (2,0)--(1,0)--(1,1)--(3,1);
            \draw[black, very thick] (2,2)--(2,0);
            \draw[black, very thick] (2,2) -- (1,1);
            \draw[black, very thick] (3,1) -- (2,0);
            \filldraw[black] (1,1) circle (2pt) node[anchor=south east]{};
            \filldraw[black] (3,1) circle (2pt) node[anchor=south west]{};
            \filldraw[black] (1,0) circle (2pt) node[anchor=south west]{};
            \filldraw[black] (2,2) circle (2pt) node[anchor=south east]{};
            \filldraw[black] (2,0) circle (2pt) node[anchor=south east]{};
        \end{tikzpicture}
    $$
    \caption{}
    \label{fig:loop mirror}
    \end{subfigure}
    \begin{subfigure}{0.3\textwidth}
    $$
        \begin{tikzpicture}[scale = 0.8]
            \draw[step=1.0,gray,thin] (0.5,-0.5) grid (3.5,2.5);
            \draw[black, very thick] (1,1) -- (3,1) -- (3,2) -- (2,2) -- (2,0) --(1,0) -- (1,1);
            \draw[black, very thick] (2,2) -- (1,1);
            \draw[black, very thick] (3,1) -- (2,0);
            \filldraw[black] (1,1) circle (2pt) node[anchor=south east]{$a$};
            \filldraw[black] (3,1) circle (2pt) node[anchor=south west]{$b$};
            \filldraw[black] (3,2) circle (2pt) node[anchor=south west]{$c$};
            \filldraw[black] (2,2) circle (2pt) node[anchor=south east]{$d$};
            \filldraw[black] (2,0) circle (2pt) node[anchor=south east]{$e$};
            \filldraw[black] (1,0) circle (2pt) node[anchor=south east]{$f$};
        \end{tikzpicture}
    $$
    \caption{}
    \label{fig:special shape}
    \end{subfigure}
    \caption{A loop from $a$ to $e$ in (\textsc{a}), its mirror in (\textsc{b}) and a special shape in (\textsc{c}).}
\end{figure}
\subsection{Knots of width \texorpdfstring{$1$}{1}}\label{subsection:width 1}
Knots of width $1$ are also called Floer homologically \emph{thin} and their knot Floer complexes over $\F[U,V]$ were completely described in \cite[Lemma 7]{petkova2013cables} by the following theorem.
\begin{theorem}\label{thm:width 1 knots}
Let $K \subset S^3$ be a knot with $w(K) = 1$. Then $\CFK_{\F[U,V]}(K)$ splits as a direct sum of a width $1$ staircase of even length and width $1$ squares.
\end{theorem}
For any $n \in \N$, there are exactly two different types of width $1$ staircases of even length (with $2n+1$ generators); $C(1, -1, \dots, 1, -1)$ and $C(-1, 1, \dots, -1, 1)$. They are depicted in Figure \ref{fig:width 1 staircase a} and Figure \ref{fig:width 1 staircase b} for $n=2$. The case $n=0$ corresponds to $\F[U,V]$ with the zero differential.

\subsection{Knots of width \texorpdfstring{$2$}{2}}
Knots of width $2$ are more complicated and the main result of this paper. We begin our analysis by a series of simple observations that culminate in Proposition \ref{prop:width2 over R}, a weaker version of Theorem \ref{thm:width2} stated over the ring $\mathcal{R} = \frac{\F[U,V]}{(UV)}$. The rest of the subsection is dedicated to carefully lifting the results from $\mathcal{R}$ to $\F[U,V]$.
\begin{lemma}
Let $L$ be a local system. Then the shape of $L$ contains two consecutive terms of the same sign.
\end{lemma}
\begin{proof}
Intuitively, having terms of alternating sign means that $L$ is extending in the northwest-southeast direction, so it cannot form a closed shape. More formally, assume that $L$ is a counterexample drawn in the plane and let $x \in L$. Consider the partial order on $\Z^2$ given by $(i,j) \leq (i',j')$ if $i \leq i'$ and $j \leq j'$. With respect to this order, any generator that can be reached from $x$ in at least $2$ steps is in the position of the plane that is not comparable to the position of $x$. On the other hand, since $L$ is a local system, we can start from $x$ and return to it, or $U^kV^kx$ for some $k \in \Z$ in finitely many steps. This is a contradiction, because the positions of $x$ and $U^kV^kx$ are comparable for any $k \in \Z$.
\end{proof}

Having established that the shape of any local system contains at least two consecutive terms of the same sign, we can investigate these parts of the local systems further.

\begin{lemma}\label{lemma:11}
Let $L$ be a local system such that the shape of $L$ contains a subsequence $1$, $1$ or $-1$, $-1$. Then $L$ is a width $1$ square.
\end{lemma}
\begin{proof}
There are four cases to consider, depending on whether the first arrow is horizontal or vertical and on whether the sequence is $1$, $1$ or $-1$, $-1$. They are not substantially different, so we deal with one of them and leave the rest to the reader.

Let us label the generators in question by $a$, $b$, and $c$ so that $\langle \partial b, Ua \rangle = 1$ and $\langle \partial c, Vb \rangle = 1$. Since $L$ is a chain complex over $\F[U,V]$, there must exist a generator $d \in L$ such that $a$, $b$, $c$, and $d$ form a width $1$ square.
\end{proof}

\begin{lemma}
Let $L$ be a local system such that the shape of $L$ contains a subsequence $2$, $2$ or $-2$, $-2$. Then $w(L) \geq 3$.
\end{lemma}
\begin{proof}
As in Lemma \ref{lemma:11}, this follows from Observation \ref{observation: 1}. In more detail, there are four cases to consider, but we only treat one. Let us label the generators in question by $a$, $b$ and $c$ so that $\langle \partial b, U^2a \rangle = 1$ and $\langle \partial c, V^2b \rangle = 1$. Then $\delta (a) = \delta (b) + 1 = \delta (c) +2$. It follows that $w(L) \geq \delta(a) - \delta (c) + 1 = 3$.
\end{proof}

Therefore, all local systems of width $2$ contain at least two consecutive terms $x$, $y$ of the same sign and such that $\{ |x|, |y| \} = \{1,2\}$. There are now eight cases, namely the four sequences $(2, 1)$, $(1, 2)$, $(-2, -1)$, $(-1, -2)$, each splitting into two cases depending on whether the first arrow is horizontal or vertical. The analysis of each of these cases is very similar, so we explore one in depth, leaving the rest to the reader.

\begin{lemma}\label{lemma:difficult lemma}
Let $L$ be a local system of width $2$ containing generators $a$, $b$ and $c$ such that $\langle \partial b, U^2a \rangle = 1$ and $\langle \partial c, Vb \rangle = 1$. Then $L$ is a width $2$ staircase or it contains a loop starting at $a$.
\end{lemma}
\begin{proof}
We first perform a change of basis such that all vertical and horizontal isomorphisms, with the possible exception of the isomorphism to which the arrow from $b$ to $a$ belongs, are given by the identity matrix. In other words, we put a local system $L$ into a standard form as in \cite[Proposition 3.7]{popovic2023link}. Most of the case analysis goes through without this assumption, but we will need it as we approach Figure \ref{fig:bad case picture}.

Let $d$ be the generator of $L$ that is connected to $c$ with a horizontal arrow. There are several cases depending on the length and direction of this arrow.
\begin{enumerate}
    \item There is an arrow from $d$ to $c$ of length $1$ or $2$.
    $$
        
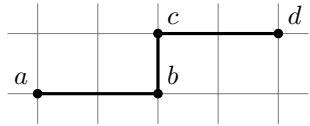
\begin{tikzpicture}[scale = 0.8]
            \draw[step=1.0,gray,thin] (0.5,0.5) grid (5.5,2.5);
            \draw[black, very thick] (1,1) -- (3,1) -- (3,2) -- (5,2);
            \filldraw[black] (1,1) circle (2pt) node[anchor=south east]{$a$};
            \filldraw[black] (3,1) circle (2pt) node[anchor=south west]{$b$};
            \filldraw[black] (3,2) circle (2pt) node[anchor=south west]{$c$};
            \filldraw[black] (5,2) circle (2pt) node[anchor=south west]{$d$};
        \end{tikzpicture}
    $$
    Such complexes cannot be lifted to chain complexes over $\F[U,V]$, so this case does not arise. 
    \item $\langle \partial c, U^2d \rangle = 1$, \emph{i.e.} there is an arrow from $c$ to $d$ of length $2$.
    $$
        
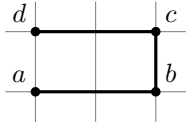
\begin{tikzpicture}[scale = 0.8]
            \draw[step=1.0,gray,thin] (0.5,0.5) grid (3.5,2.5);
            \draw[black, very thick] (1,1) -- (3,1) -- (3,2) -- (1,2);
            \filldraw[black] (1,1) circle (2pt) node[anchor=south east]{$a$};
            \filldraw[black] (3,1) circle (2pt) node[anchor=south west]{$b$};
            \filldraw[black] (3,2) circle (2pt) node[anchor=south west]{$c$};
            \filldraw[black] (1,2) circle (2pt) node[anchor=south east]{$d$};
        \end{tikzpicture}
    $$
    Since $L$ is a mod $UV$ reduction of a chain complex over $\F[U,V]$, it follows that there must be a vertical arrow from $d$ to $a$ and $L$ is equal to
    $$
        
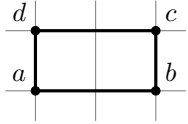
\begin{tikzpicture}[scale = 0.8]
            \draw[step=1.0,gray,thin] (0.5,0.5) grid (3.5,2.5);
            \draw[black, very thick] (1,1) -- (3,1) -- (3,2) -- (1,2) -- (1,1);
            \filldraw[black] (1,1) circle (2pt) node[anchor=south east]{$a$};
            \filldraw[black] (3,1) circle (2pt) node[anchor=south west]{$b$};
            \filldraw[black] (3,2) circle (2pt) node[anchor=south west]{$c$};
            \filldraw[black] (1,2) circle (2pt) node[anchor=south east]{$d$};
        \end{tikzpicture}
    $$
    This is a width $2$ staircase.

    \item $\langle \partial c, Ud \rangle = 1$, \emph{i.e.} there is an arrow from $c$ to $d$ of length $1$.
    $$
        
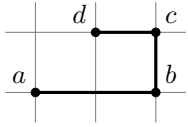
\begin{tikzpicture}[scale = 0.8]
            \draw[step=1.0,gray,thin] (0.5,0.5) grid (3.5,2.5);
            \draw[black, very thick] (1,1) -- (3,1) -- (3,2) -- (2,2);
            \filldraw[black] (1,1) circle (2pt) node[anchor=south east]{$a$};
            \filldraw[black] (3,1) circle (2pt) node[anchor=south west]{$b$};
            \filldraw[black] (3,2) circle (2pt) node[anchor=south west]{$c$};
            \filldraw[black] (2,2) circle (2pt) node[anchor=south east]{$d$};
        \end{tikzpicture}
    $$
    Since $\partial^2 c=0$, there must be a diagonal arrow from $d$ to $a$. Consider the vertical arrow adjacent to $d$ and let $e$ be its other endpoint. If the arrow is pointing from $d$ to $e$, then it must have length $2$ since $\partial^2 c = 0$ cannot be established otherwise, \emph{cf.} Lemma \ref{lemma:11}. This in turn forces a diagonal arrow from $b$ to $e$ and we have established a loop in $L$ starting at $a$.
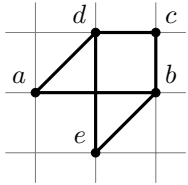
    $$
        \begin{tikzpicture}[scale = 0.8]
            \draw[step=1.0,gray,thin] (0.5,-0.5) grid (3.5,2.5);
            \draw[black, very thick] (1,1) -- (3,1) -- (3,2) -- (2,2) -- (2,0);
            \draw[black, very thick] (2,2) -- (1,1);
            \draw[black, very thick] (3,1) -- (2,0);
            \filldraw[black] (1,1) circle (2pt) node[anchor=south east]{$a$};
            \filldraw[black] (3,1) circle (2pt) node[anchor=south west]{$b$};
            \filldraw[black] (3,2) circle (2pt) node[anchor=south west]{$c$};
            \filldraw[black] (2,2) circle (2pt) node[anchor=south east]{$d$};
            \filldraw[black] (2,0) circle (2pt) node[anchor=south east]{$e$};
        \end{tikzpicture}
    $$
    The remaining case is when the arrow is pointing from $e$ to $d$. Its length must be $1$ since otherwise $\delta(e)-\delta(a) \geq 2$ and $w(L) \geq 3$. Since $\partial^2 e = 0$, there must be an even number of ways of traveling from $e$ to $a$ in two steps. Because it is possible to pass through $d$, there must be another path through a new generator $f$ and there are only two possible locations for this generator. They are depicted on the pictures below.
    $$
        
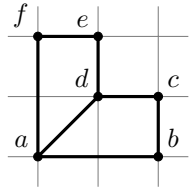
\begin{tikzpicture}[scale = 0.8]
            \draw[step=1.0,gray,thin] (0.5, 0.5) grid (3.5,3.5);
            \draw[black, very thick] (1,1) -- (3,1) -- (3,2) -- (2,2) -- (2,3) -- (1,3) -- (1,1);
            \draw[black, very thick] (2,2) -- (1,1);
            \filldraw[black] (1,1) circle (2pt) node[anchor=south east]{$a$};
            \filldraw[black] (3,1) circle (2pt) node[anchor=south west]{$b$};
            \filldraw[black] (3,2) circle (2pt) node[anchor=south west]{$c$};
            \filldraw[black] (2,3) circle (2pt) node[anchor=south east]{$e$};
            \filldraw[black] (2,2) circle (2pt) node[anchor=south east]{$d$};
            \filldraw[black] (1,3) circle (2pt) node[anchor=south east]{$f$};
        \end{tikzpicture}
        
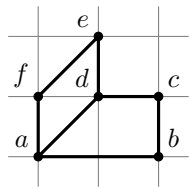
\hspace{2cm}
        \begin{tikzpicture}[scale = 0.8]
            \draw[step=1.0,gray,thin] (0.5, 0.5) grid (3.5,3.5);
            \draw[black, very thick] (1,1) -- (3,1) -- (3,2) -- (2,2) -- (2,3) -- (1,2) -- (1,1);
            \draw[black, very thick] (2,2) -- (1,1);
            \filldraw[black] (1,1) circle (2pt) node[anchor=south east]{$a$};
            \filldraw[black] (3,1) circle (2pt) node[anchor=south west]{$b$};
            \filldraw[black] (3,2) circle (2pt) node[anchor=south west]{$c$};
            \filldraw[black] (2,3) circle (2pt) node[anchor=south east]{$e$};
            \filldraw[black] (2,2) circle (2pt) node[anchor=south east]{$d$};
            \filldraw[black] (1,2) circle (2pt) node[anchor=south east]{$f$};
        \end{tikzpicture}
    $$
    The left picture is a width $2$ staircase, so we are done. In the right picture, the generators $e$ and $f$ possess a property we wish to axiomatize.
    \begin{definition}
        A pair of generators $x, y \in C$ with $\langle \partial x, UVy \rangle = 1$ and outgoing vertical arrows is called a \emph{crest pair}. A pair of generators $x, y \in C$ with $\langle \partial x, UVy \rangle = 1$ and incoming horizontal arrows is called a \emph{trough pair}.
    \end{definition}
    A crest pair and a trough pair are schematically drawn in Figure \ref{fig:crest and trough pairs}.
    
\begin{figure}
\centering
\begin{subfigure}{0.4\textwidth}
    $$
        \begin{tikzpicture}[scale = 1.0]
            \draw[step=1.0,gray,thin] (0.5, 1.1) grid (2.9,3.5);
            \draw[black, very thick] (1,2) -- (1,1.5);
            \draw[black, very thick] (2,3) -- (2,2.5);
            \draw[black, very thick] (2,3) -- (1,2);
            \filldraw[black] (2,3) circle (2pt) node[anchor=south east]{$x$};
            \filldraw[black] (1,2) circle (2pt) node[anchor=south east]{$y$};
        \end{tikzpicture}
    $$
    \caption{A crest pair}
    \label{fig:crest pair}
\end{subfigure}
\begin{subfigure}{0.4\textwidth}
    $$
        \begin{tikzpicture}[scale = 1.0]
            \draw[step=1.0,gray,thin] (0.5, 1.1) grid (2.9,3.5);
            \draw[black, very thick] (1,2) -- (1.5,2);
            \draw[black, very thick] (2,3) -- (2.5,3);
            \draw[black, very thick] (2,3) -- (1,2);
            \filldraw[black] (2,3) circle (2pt) node[anchor=south east]{$x$};
            \filldraw[black] (1,2) circle (2pt) node[anchor=south east]{$y$};
        \end{tikzpicture}
    $$
    \caption{A trough pair}
    \label{fig:trough pair}
\end{subfigure}
\caption{}
\label{fig:crest and trough pairs}
\end{figure}

\vspace{1em}

Resuming the argument, let $x, y \in L$ be a crest pair such that there is a path consisting of horizontal and vertical segments connecting $x$ and $y$. Consider the horizontal arrow adjacent to $y$ and denote its other endpoint by $z$. Assume first that the arrow goes from $y$ to $z$. In that case, it must be of length $1$ since otherwise $\delta (x) - \delta(z) \geq 2$. Since $\partial^2 x=0$, there is an even number of ways of traveling from $x$ to $z$. Because one path passes through $y$, there must exist another path passing through a different generator $w$. The two possible locations for $w$ are depicted on the following pictures.
$$
        
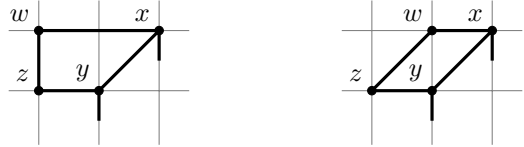
\begin{tikzpicture}[scale = 0.8]
            \draw[step=1.0,gray,thin] (-0.5, 1.1) grid (2.5,3.5);
            \draw[black, very thick] (1,2) -- (1,1.5);
            \draw[black, very thick] (2,3) -- (2,2.5);
            \draw[black, very thick] (2,3) -- (1,2);
            \draw[black, very thick] (1,2) -- (0,2) -- (0,3) -- (2,3);
            \filldraw[black] (2,3) circle (2pt) node[anchor=south east]{$x$};
            \filldraw[black] (1,2) circle (2pt) node[anchor=south east]{$y$};
            \filldraw[black] (0,2) circle (2pt) node[anchor=south east]{$z$};
            \filldraw[black] (0,3) circle (2pt) node[anchor=south east]{$w$};
        \end{tikzpicture}
        \hspace{2cm}
        \begin{tikzpicture}[scale = 0.8]
            \draw[step=1.0,gray,thin] (-0.5, 1.1) grid (2.5,3.5);
            \draw[black, very thick] (1,2) -- (1,1.5);
            \draw[black, very thick] (2,3) -- (2,2.5);
            \draw[black, very thick] (2,3) -- (1,2);
            \draw[black, very thick] (1,2) -- (0,2) -- (1,3) -- (2,3);
            \filldraw[black] (2,3) circle (2pt) node[anchor=south east]{$x$};
            \filldraw[black] (1,2) circle (2pt) node[anchor=south east]{$y$};
            \filldraw[black] (0,2) circle (2pt) node[anchor=south east]{$z$};
            \filldraw[black] (1,3) circle (2pt) node[anchor=south east]{$w$};
        \end{tikzpicture}
    $$
    The case on the left completes the local system and we obtain a width $2$ staircase. The case on the right provides us with the trough pair $z$, $w$ and a similar analysis can be carried out by considering the vertical arrow adjacent to $w$. Provided it is incoming, the local system will either be completed into a width $2$ staircase or we will obtain a new crest pair. Since $L$ is finitely generated, the crest and trough pairs will stop alternating at some point and $L$ will be completed into a width $2$ staircase.

    \hspace{1em}

    The above process assumes that there is always an outgoing horizontal arrow from the lower generator of a crest pair and an incoming vertical arrow to the upper generator of the trough pair. Let us now investigate what happens if this is not the case, \emph{i.e.}, if for example at some point the lower generator of a crest pair has an incoming horizontal arrow. Such a scenario is depicted in Figure \ref{fig:bad case picture}. The crest pair generators are denoted by $x$ and $y$ and there is a horizontal arrow from $z$ to $y$. Such an arrow must necessarily be of length $2$ since otherwise $\partial^2z=0$ cannot be established by Lemma \ref{lemma:11}.
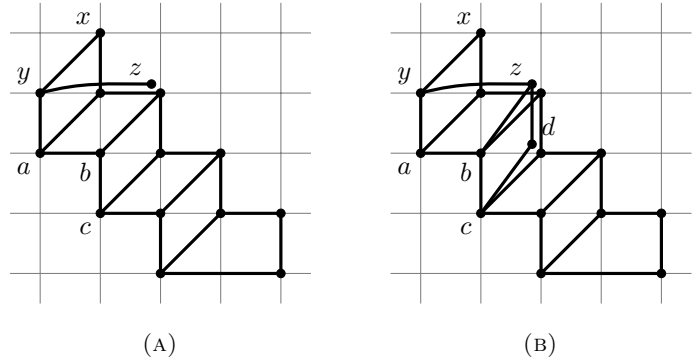
\begin{figure}
\centering
\begin{subfigure}{0.4\textwidth}
    $$
    \begin{tikzpicture}[scale = 0.8]
            \def\a{0.15}
            \draw[step=1.0,gray,thin] (-1.5, 0.5) grid (3.5,5.5);
            \draw[black, very thick] (-1,4) -- (-1,3) -- (0,3) -- (0,2) -- (1,2) -- (1,1) -- (3,1) -- (3,2) -- (2,2) -- (2,3) -- (1,3) -- (1,4) -- (0,4) -- (0,5);
            \draw[black, very thick] (2,2) -- (1,1);
            \draw[black, very thick] (2,3) -- (1,2);
            \draw[black, very thick] (1,3) -- (0,2);
            \draw[black, very thick] (1,4) -- (0,3);
            \draw[black, very thick] (0,4) -- (-1,3);
            \draw[black, very thick] (0,5) -- (-1,4);
            \draw[black, very thick] (-1,4) to[out=15, in=180] (1-\a, 4+\a);
            \filldraw[black] (-1,4) circle (2pt) node[anchor=south east]{$y$};
            \filldraw[black] (-1,3) circle (2pt) node[anchor=north east]{$a$};
            \filldraw[black] (0,3) circle (2pt) node[anchor=north east]{$b$};
            \filldraw[black] (0,2) circle (2pt) node[anchor=north east]{$c$};
            \filldraw[black] (1,2) circle (2pt) node[anchor=south east]{};
            \filldraw[black] (1,1) circle (2pt) node[anchor=south east]{};
            \filldraw[black] (3,1) circle (2pt) node[anchor=south east]{};
            \filldraw[black] (3,2) circle (2pt) node[anchor=south east]{};
            \filldraw[black] (2,2) circle (2pt) node[anchor=south east]{};
            \filldraw[black] (2,3) circle (2pt) node[anchor=south east]{};
            \filldraw[black] (1,3) circle (2pt) node[anchor=south east]{};
            \filldraw[black] (1,4) circle (2pt) node[anchor=south east]{};
            \filldraw[black] (0,4) circle (2pt) node[anchor=south east]{};
            \filldraw[black] (0,5) circle (2pt) node[anchor=south east]{$x$};
            \filldraw[black] (1-\a,4+\a) circle (2pt) node[anchor=south east]{$z$};
    \end{tikzpicture}
    $$
    \caption{} 
    \label{fig:bad case picture}
\end{subfigure}
\begin{subfigure}{0.4\textwidth}
    $$
    \begin{tikzpicture}[scale = 0.8]
            \def\a{0.15}
            \draw[step=1.0,gray,thin] (-1.5, 0.5) grid (3.5,5.5);
            \draw[black, very thick] (-1,4) -- (-1,3) -- (0,3) -- (0,2) -- (1,2) -- (1,1) -- (3,1) -- (3,2) -- (2,2) -- (2,3) -- (1,3) -- (1,4) -- (0,4) -- (0,5);
            \draw[black, very thick] (2,2) -- (1,1);
            \draw[black, very thick] (2,3) -- (1,2);
            \draw[black, very thick] (1,3) -- (0,2);
            \draw[black, very thick] (1,4) -- (0,3);
            \draw[black, very thick] (0,4) -- (-1,3);
            \draw[black, very thick] (0,5) -- (-1,4);
            \draw[black, very thick] (-1,4) to[out=15, in=180] (1-\a, 4+\a);
            \draw[black, very thick] (0,3) -- (1-\a, 4+\a);
            \draw[black, very thick] (1-\a,3+\a) -- (1-\a, 4+\a);
            \draw[black, very thick] (0,2) -- (1-\a, 3+\a);
            \filldraw[black] (-1,4) circle (2pt) node[anchor=south east]{$y$};
            \filldraw[black] (-1,3) circle (2pt) node[anchor=north east]{$a$};
            \filldraw[black] (0,3) circle (2pt) node[anchor=north east]{$b$};
            \filldraw[black] (0,2) circle (2pt) node[anchor=north east]{$c$};
            \filldraw[black] (1,2) circle (2pt) node[anchor=south east]{};
            \filldraw[black] (1,1) circle (2pt) node[anchor=south east]{};
            \filldraw[black] (3,1) circle (2pt) node[anchor=south east]{};
            \filldraw[black] (3,2) circle (2pt) node[anchor=south east]{};
            \filldraw[black] (2,2) circle (2pt) node[anchor=south east]{};
            \filldraw[black] (2,3) circle (2pt) node[anchor=south east]{};
            \filldraw[black] (1,3) circle (2pt) node[anchor=south east]{};
            \filldraw[black] (1,4) circle (2pt) node[anchor=south east]{};
            \filldraw[black] (0,4) circle (2pt) node[anchor=south east]{};
            \filldraw[black] (0,5) circle (2pt) node[anchor=south east]{$x$};
            \filldraw[black] (1-\a,4+\a) circle (2pt) node[anchor=south east]{$z$};
            \filldraw[black] (1-\a,3+\a) circle (2pt) node[anchor=south west]{$d$};
    \end{tikzpicture}
    $$
    \caption{} 
    \label{fig:bad case continuation}
\end{subfigure}
\caption{Subfigure (\textsc{a}) demonstrates the case of Lemma \ref{lemma:difficult lemma} in which the lower generator of a crest pair $x,y$ has an incoming horizontal arrow. Subfigure (\textsc{b}) is a deeper analysis of this case that eventually shows that such a case does not arise.}
\label{fig:bad}
\end{figure}
Let $a, b, c \in L$ be the generators as labelled on Figure $\ref{fig:bad case picture}$. Because $\partial^2 z = 0$, there must be an even number of paths from $z$ to $a$ of length $2$. This forces a diagonal arrow from $z$ to $b$. Here we are using the fact that $L$ is horizontally simplified away from the bottom edge. There must also be an even number of paths from $z$ to $c$ of length $2$, which forces the existence of a new generator $d$ together with a vertical arrow from $z$ to $d$ and a diagonal arrow from $d$ to $c$. The situation is depicted in Figure \ref{fig:bad case continuation}. Note again that since $L$ is vertically simplified, $d$ cannot be the generator that is already drawn. Now the horizontal arrow adjacent to $d$ must be incoming and of length $1$ and one can continue with the analysis until one reaches the ``end" of the staircase as shown in Figure \ref{fig:bad case end-1}.
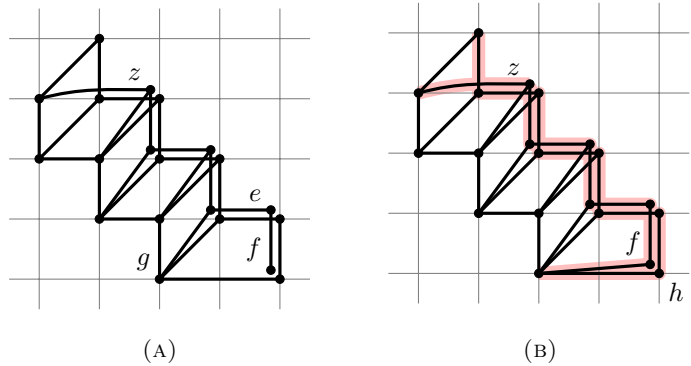
\begin{figure}
\centering
\begin{subfigure}{0.4\textwidth}
    $$
    \begin{tikzpicture}[scale = 0.8]
            \def\a{0.15}
            \draw[step=1.0,gray,thin] (-1.5, 0.5) grid (3.5,5.5);
            \draw[black, very thick] (-1,4) -- (-1,3) -- (0,3) -- (0,2) -- (1,2) -- (1,1) -- (3,1) -- (3,2) -- (2,2) -- (2,3) -- (1,3) -- (1,4) -- (0,4) -- (0,5);
            \draw[black, very thick] (2,2) -- (1,1);
            \draw[black, very thick] (2,3) -- (1,2);
            \draw[black, very thick] (1,3) -- (0,2);
            \draw[black, very thick] (1,4) -- (0,3);
            \draw[black, very thick] (0,4) -- (-1,3);
            \draw[black, very thick] (0,5) -- (-1,4);
            \draw[black, very thick] (-1,4) to[out=15, in=180] (1-\a, 4+\a);
            \draw[black, very thick] (0,3) -- (1-\a, 4+\a);
            \draw[black, very thick] (1-\a,3+\a) -- (1-\a, 4+\a);
            \draw[black, very thick] (1-\a,3+\a) -- (2-\a, 3+\a);
            \draw[black, very thick] (2-\a,3+\a) -- (2-\a, 2+\a);
            \draw[black, very thick] (2-\a,2+\a) -- (3-\a, 2+\a);
            \draw[black, very thick] (3-\a,2+\a) -- (3-\a, 1+\a);
            \draw[black, very thick] (0,2) -- (1-\a, 3+\a);

            \draw[black, very thick] (1,2) -- (2-\a, 3+\a);
            \draw[black, very thick] (1,1) -- (2-\a, 2+\a);

            \filldraw[black] (-1,4) circle (2pt) node[anchor=south east]{};
            \filldraw[black] (-1,3) circle (2pt) node[anchor=north east]{};
            \filldraw[black] (0,3) circle (2pt) node[anchor=north east]{};
            \filldraw[black] (0,2) circle (2pt) node[anchor=north east]{};
            \filldraw[black] (1,2) circle (2pt) node[anchor=south east]{};
            \filldraw[black] (1,1) circle (2pt) node[anchor=south east]{$g$};
            \filldraw[black] (3,1) circle (2pt) node[anchor=south east]{};
            \filldraw[black] (3,2) circle (2pt) node[anchor=south east]{};
            \filldraw[black] (2,2) circle (2pt) node[anchor=south east]{};
            \filldraw[black] (2,3) circle (2pt) node[anchor=south east]{};
            \filldraw[black] (1,3) circle (2pt) node[anchor=south east]{};
            \filldraw[black] (1,4) circle (2pt) node[anchor=south east]{};
            \filldraw[black] (0,4) circle (2pt) node[anchor=south east]{};
            \filldraw[black] (0,5) circle (2pt) node[anchor=south east]{};
            \filldraw[black] (1-\a,4+\a) circle (2pt) node[anchor=south east]{$z$};
            \filldraw[black] (1-\a,3+\a) circle (2pt) node[anchor=south west]{};
            \filldraw[black] (2-\a,3+\a) circle (2pt) node[anchor=south west]{};
            \filldraw[black] (2-\a,2+\a) circle (2pt) node[anchor=south west]{};
            \filldraw[black] (3-\a,2+\a) circle (2pt) node[anchor=south east]{$e$};
            \filldraw[black] (3-\a,1+\a) circle (2pt) node[anchor=south east]{$f$};
    \end{tikzpicture}
    $$
    \caption{} 
    \label{fig:bad case end-1}
\end{subfigure}
\begin{subfigure}{0.4\textwidth}
    $$
    \begin{tikzpicture}[scale = 0.8]
            \def\a{0.15}
            \draw[line width = 5pt, line join=round, red!25!white] (1,1) -- (3,1) -- (3,2) -- (2,2) -- (2,3) -- (1,3) -- (1,4) -- (0,4) -- (0,5);
            \draw[line width = 5pt, line join=round, red!25!white] (1,1) -- (3-\a,1+\a) -- (3-\a,2+\a) -- (2-\a,2+\a) -- (2-\a,3+\a) -- (1-\a,3+\a) -- (1-\a,4+\a);
            \draw[line width = 5pt, line join=round, red!25!white] (-1,4) to[out=15, in=180] (1-\a, 4+\a);
            \draw[step=1.0,gray,thin] (-1.5, 0.5) grid (3.5,5.5);
            \draw[black, very thick] (-1,4) -- (-1,3) -- (0,3) -- (0,2) -- (1,2) -- (1,1) -- (3,1) -- (3,2) -- (2,2) -- (2,3) -- (1,3) -- (1,4) -- (0,4) -- (0,5);
            \draw[black, very thick] (2,2) -- (1,1);
            \draw[black, very thick] (2,3) -- (1,2);
            \draw[black, very thick] (1,3) -- (0,2);
            \draw[black, very thick] (1,4) -- (0,3);
            \draw[black, very thick] (0,4) -- (-1,3);
            \draw[black, very thick] (0,5) -- (-1,4);
            \draw[black, very thick] (-1,4) to[out=15, in=180] (1-\a, 4+\a);
            \draw[black, very thick] (0,3) -- (1-\a, 4+\a);
            \draw[black, very thick] (1-\a,3+\a) -- (1-\a, 4+\a);
            \draw[black, very thick] (1-\a,3+\a) -- (2-\a, 3+\a);
            \draw[black, very thick] (2-\a,3+\a) -- (2-\a, 2+\a);
            \draw[black, very thick] (2-\a,2+\a) -- (3-\a, 2+\a);
            \draw[black, very thick] (3-\a,2+\a) -- (3-\a, 1+\a);
            \draw[black, very thick] (0,2) -- (1-\a, 3+\a);
            \draw[black, very thick] (1,2) -- (2-\a, 3+\a);
            \draw[black, very thick] (1,1) -- (2-\a, 2+\a);
            \draw[black, very thick] (1,1) -- (3-\a, 1+\a);
            
            \filldraw[black] (-1,4) circle (2pt) node[anchor=south east]{};
            \filldraw[black] (-1,3) circle (2pt) node[anchor=north east]{};
            \filldraw[black] (0,3) circle (2pt) node[anchor=north east]{};
            \filldraw[black] (0,2) circle (2pt) node[anchor=north east]{};
            \filldraw[black] (1,2) circle (2pt) node[anchor=south east]{};
            \filldraw[black] (1,1) circle (2pt) node[anchor=south east]{};
            \filldraw[black] (3,1) circle (2pt) node[anchor=north west]{$h$};
            \filldraw[black] (3,2) circle (2pt) node[anchor=south east]{};
            \filldraw[black] (2,2) circle (2pt) node[anchor=south east]{};
            \filldraw[black] (2,3) circle (2pt) node[anchor=south east]{};
            \filldraw[black] (1,3) circle (2pt) node[anchor=south east]{};
            \filldraw[black] (1,4) circle (2pt) node[anchor=south east]{};
            \filldraw[black] (0,4) circle (2pt) node[anchor=south east]{};
            \filldraw[black] (0,5) circle (2pt) node[anchor=south east]{};
            \filldraw[black] (1-\a,4+\a) circle (2pt) node[anchor=south east]{$z$};
            \filldraw[black] (1-\a,3+\a) circle (2pt) node[anchor=south west]{};
            \filldraw[black] (2-\a,3+\a) circle (2pt) node[anchor=south west]{};
            \filldraw[black] (2-\a,2+\a) circle (2pt) node[anchor=south west]{};
            \filldraw[black] (3-\a,2+\a) circle (2pt) node[anchor=south west]{};
            \filldraw[black] (3-\a,1+\a) circle (2pt) node[anchor=south east]{$f$};
    \end{tikzpicture}
    $$
    \caption{} 
    \label{fig:bad case end}
\end{subfigure}
\caption{Subfigure (\textsc{a}) depicts continuation of the analysis from Figure \ref{fig:bad} to the point when one reaches the end of the staircase. Subfigure (\textsc{b}) shows the ensuing contradiction, since the strands of generators $f$ and $h$ eventually diverge near $z$.}
\end{figure}
Since $\partial^2e=0$ is required, there is another way of traveling from $e$ to $g$ in two steps. But $e$ is already adjacent to its unique horizontal and vertical arrows and $g$ is already adjacent to its unique vertical arrow. We emphasize that, while $g$ is not necessarily adjacent to only one horizontal arrow, since the bottom edge of the figure is not necessarily horizontally simplified, all horizontal arrows at $g$ have the same direction and length. Therefore, the alternative way of traveling from $e$ to $g$ in two steps passes through $f$ as shown in Figure \ref{fig:bad case end}. However, this is not how local systems work, since the strands of generators $f$ and $h$ eventually diverge near $z$. This is a contradiction which shows that such a scenario cannot happen in a local system.

The analysis of the case in which there is an outgoing vertical arrow from the upper generator of a trough pair is entirely analogous and shows that such a situation cannot arise either.
\end{enumerate}
\end{proof}
With the help of Lemma \ref{lemma:difficult lemma}, we are now able to classify local systems that appear in width $2$ knot Floer complexes.
\begin{proposition}\label{prop:width2 over R}
    Let $K \subset S^3$ be a knot with $w(K)=2$ and let $\CFK_{\F[U,V]}(K)$ be its knot Floer complex. Then $\CFK_{\mathcal{R}}(K)$ splits uniquely as a direct sum of a width $2$ standard complex and some trivial local systems of the following shapes:
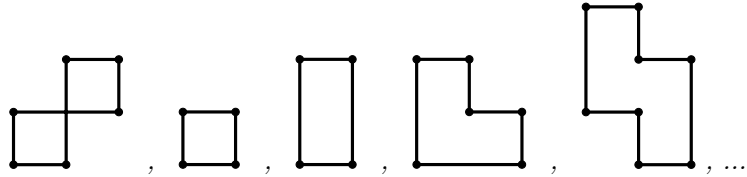
\begin{center}
\begin{tikzpicture}[scale = 0.7]
    \draw[black, very thick] (1,0) -- (1,1) -- (3,1) -- (3,2) -- (2,2) -- (2,0) -- (1,0);
    \filldraw[black] (1,1) circle (2pt) node[anchor=south east]{};
    \filldraw[black] (3,1) circle (2pt) node[anchor=south west]{};
    \filldraw[black] (3,2) circle (2pt) node[anchor=south west]{};
    \filldraw[black] (2,2) circle (2pt) node[anchor=south east]{};
    \filldraw[black] (2,0) circle (2pt) node[anchor=south east]{};
    \filldraw[black] (1,0) circle (2pt) node[anchor=south east]{};
\end{tikzpicture}
,
\begin{tikzpicture}[scale = 0.7]
    \draw[black, very thick] (1,1) -- (2,1) -- (2,2) -- (1,2) -- (1,1);
    \filldraw[black] (1,1) circle (2pt) node[anchor=south east]{};
    \filldraw[black] (2,1) circle (2pt) node[anchor=south west]{};
    \filldraw[black] (2,2) circle (2pt) node[anchor=south west]{};
    \filldraw[black] (1,2) circle (2pt) node[anchor=south east]{};
\end{tikzpicture}
,
\begin{tikzpicture}[scale = 0.7]
    \draw[black, very thick] (1,1) -- (2,1) -- (2,3) -- (1,3) -- (1,1);
    \filldraw[black] (1,1) circle (2pt) node[anchor=south east]{};
    \filldraw[black] (2,1) circle (2pt) node[anchor=south west]{};
    \filldraw[black] (2,3) circle (2pt) node[anchor=south west]{};
    \filldraw[black] (1,3) circle (2pt) node[anchor=south east]{};
\end{tikzpicture}
,
\begin{tikzpicture}[scale = 0.7]
    \draw[black, very thick] (1,1) -- (3,1) -- (3,2) -- (2,2) -- (2,3) -- (1,3) -- (1,1);
    \filldraw[black] (1,1) circle (2pt) node[anchor=south east]{};
    \filldraw[black] (3,1) circle (2pt) node[anchor=south west]{};
    \filldraw[black] (3,2) circle (2pt) node[anchor=south west]{};
    \filldraw[black] (2,3) circle (2pt) node[anchor=south east]{};
    \filldraw[black] (2,2) circle (2pt) node[anchor=south east]{};
    \filldraw[black] (1,3) circle (2pt) node[anchor=south east]{};
\end{tikzpicture}
,
\begin{tikzpicture}[scale = 0.7]
    \draw[black, very thick] (0,0)--(1,0)--(1,-1)--(2,-1)--(2,-3)--(1,-3)--(1,-2)--(0,-2)--(0,0);
    \filldraw[black] (0,0) circle (2pt) node[anchor=south east]{};
    \filldraw[black] (1,0) circle (2pt) node[anchor=south east]{};
    \filldraw[black] (1,-1) circle (2pt) node[anchor=south east]{};
    \filldraw[black] (2,-1) circle (2pt) node[anchor=south east]{};
    \filldraw[black] (2,-3) circle (2pt) node[anchor=south east]{};
    \filldraw[black] (1,-3) circle (2pt) node[anchor=south east]{};
    \filldraw[black] (1,-2) circle (2pt) node[anchor=south east]{};
    \filldraw[black] (0,-2) circle (2pt) node[anchor=south east]{};
\end{tikzpicture}
, ...
\end{center}
and their reflections.
\end{proposition}
Note that the statement of Proposition \ref{prop:width2 over R} is almost identical to the statement of Theorem \ref{thm:width2}. The only difference is in the choice of the ring we are working over. In the case of the proposition, this is $\mathcal{R} = \frac{\F[U,V]}{(UV)}$, which does not \emph{see} the diagonal arrows in the full complex $\CFK_{\F[U,V]}(K)$. As such, Proposition \ref{prop:width2 over R} is a weaker result than Theorem \ref{thm:width2}, but nonetheless provides a good stepping stone towards the general result. 
\begin{proof}[Proof of Proposition \ref{prop:width2 over R}]
Every $\CFK_{\mathcal{R}}(K)$ admits a lift to $\CFK_{\F[U,V]}(K)$ so we will work over $\F[U,V]$ throughout most of the proof and only pass to $\mathcal{R}$ at the end. Let $L$ be a local system of width $2$ and assume it is depicted in the plane in a horizontally simplified basis. By the remark preceding Lemma \ref{lemma:difficult lemma}, one of the eight similar cases arises. We have explored one of them in depth in Lemma \ref{lemma:difficult lemma} and noted that all others allow for a similar conclusion. We are done if $L$ is a width $2$ staircase. If not, then $L$ contains a loop, say a loop from $a$ to $e$. Let $f$ be the other endpoint of a horizontal arrow adjacent to $e$. If there is an arrow from $e$ to $f$, then it must have length $1$ since otherwise $w(L) \geq \delta(d) - \delta(f) +1 \geq 3$. Since $\partial^2 b = 0$ and $L$ is horizontally simplified, the only possible alternative way of traveling from $b$ to $f$ in two steps passes through $a$. In other words, we must have $\partial a= Vf$ and the local system is the special shape. Otherwise there is an arrow from $f$ to $e$. Note that the generator $a$ is above and to the left of $e$ and we claim that its position in the plane cannot be reached again. To see that, travel from $f$ in the direction away from $a$ via vertical and horizontal arrows until the first arrow $A$ that points upwards or to the left. By our assumption, it is preceded by an arrow $B$ pointing to the right or downwards. In either case, we see that $A$ or $B$ has length $2$ and that the arrows $A$ and $B$ satisfy the conditions of Lemma \ref{lemma:difficult lemma}. Because $L$ can no longer be a width $2$ staircase, we either obtain a contradiction or another loop, whose final endpoint is further down and to the right. It follows that none of the generators $U^kV^ka$ for $k \in \Z$ can be reached again. Therefore, all local systems of width $2$ are width $2$ staircases and the special shape. Forgetting the diagonal arrows now yields the required classification.
\end{proof}
Proposition \ref{prop:width2 over R} describes the direct summands of $\CFK_{\mathcal{R}}(K)$. The knot Floer complex can be lifted to $\CFK_{\F[U,V]}(K)$, but the decomposition into indecomposable summands might not be identical. In other words, different local systems over $\mathcal{R}$ might become connected over $\F[U,V]$ in the presence of diagonal arrows. We call such diagonal arrows between different local systems \emph{external}. By finding a suitable change of basis, we show in Proposition \ref{prop:disconnect everything} that all external arrows can be removed, thus lifting the results of Proposition \ref{prop:width2 over R} from $\mathcal{R}$ to $\F[U,V]$.

\vspace{1em}

We begin with some notation. Let $K_0, K_1, \dots$ denote the trivial local systems as drawn in the statement of Theorem \ref{thm:width2} -- $K_0$ is the special shape, $K_1$ is a $1 \times 1$ square, $K_2$ is a $2 \times 1$ rectangle, and so on. As a useful mnemonic, one can observe that for $n\geq 1$, $K_n$ has area $n$. Let $\K{0}, \K{1}, \K{2}, \dots$ be their reflections over the line $y=-x$. Note that $K_0 = \K{0}$ and $K_1 = \K{1}$, while the rest of the local systems $K_n$ are not symmetric over this line.
\begin{proposition}\label{prop:disconnect everything}
    Let $C$ be a chain complex over $\F[U,V]$ of width $2$ and let $L \in \{K_0, K_1,  \dots, \} \cup \{ \K{0}, \K{1},  \dots \}$ be a local system such that $C \cong L \oplus C'$ over $\mathcal{R}$ for some chain complex $C'$. Then $C \cong L \oplus C'$ over $\F[U,V]$.
\end{proposition}
\begin{proof}
    This is the content of Lemma \ref{lemma:remove K_1}, Lemma \ref{lemma:remove K_0} and Lemma \ref{lemma:remove the rest}.
\end{proof}

\begin{lemma}\label{lemma:remove K_1}
If $C \cong K_1 \oplus C'$ over $\mathcal{R}$, then $C \cong K_1 \oplus C'$ over $\F[U,V]$.
\end{lemma}
In other words, a $1 \times 1$ square can be completely disconnected from the rest of the complex, even in the presence of diagonal arrows.
\begin{proof}
Let the square $K_1$ have vertices $a, b, c, d$ as depicted.
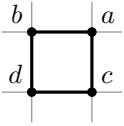
\begin{center}
\begin{tikzpicture}[scale=0.8]
    \draw[step=1.0,gray,thin] (-0.5, -1.5) grid (1.5,0.5);
    \draw[very thick] (0,0)--(1,0)--(1,-1)--(0,-1)--(0,0);
    \filldraw[black] (1,0) circle (2pt) node[anchor=south west]{$a$};
    \filldraw[black] (0,0) circle (2pt) node[anchor=south east]{$b$};
    \filldraw[black] (1,-1) circle (2pt) node[anchor=south west]{$c$};
    \filldraw[black] (0,-1) circle (2pt) node[anchor=south east]{$d$};
\end{tikzpicture}
\end{center}
If there are no diagonal arrows adjacent to any of the vertices of the square, then we are already done. Otherwise, there are $8$ cases depending on which of the square vertices a diagonal arrow is adjacent to and what the direction of this arrow is. Note that the four cases with outgoing arrows are symmetric to the four cases with incoming arrows. Since the presented analysis can easily be adapted to a symmetric case, it is sufficient to deal with the four cases involving incoming diagonal arrows to $a, b, c, d$.
\begin{enumerate}
    \item There is an incoming arrow at $a$.

    Let $e$ be the other endpoint of the arrow. Since $\partial^2 e = 0$, there is an even number of ways of traveling from $e$ to $b$ and from $e$ to $c$ in two steps. This means that besides passing through $a$, there is another way of traveling from $e$ to $b$ in two steps and another way of traveling from $e$ to $c$ in two steps. These two ways must be as depicted.
    
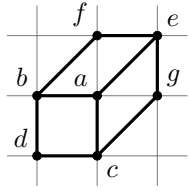
\begin{center}
    \begin{tikzpicture}[scale=0.8]
        \draw[step=1.0,gray,thin] (-0.5, -1.5) grid (2.5,1.5);
        \draw[very thick] (0,0)--(1,0)--(1,-1)--(0,-1)--(0,0);
        \draw[very thick] (1,0)--(2,1);
        \draw[very thick] (0,0)--(1,1);
        \draw[very thick] (1,-1)--(2,0);
        \draw[very thick] (1,1)--(2,1)--(2,0);
        \filldraw[black] (1,0) circle (2pt) node[anchor=south east]{$a$};
        \filldraw[black] (0,0) circle (2pt) node[anchor=south east]{$b$};
        \filldraw[black] (1,-1) circle (2pt) node[anchor=north west]{$c$};
        \filldraw[black] (0,-1) circle (2pt) node[anchor=south east]{$d$};
        \filldraw[black] (2,1) circle (2pt) node[anchor=south west]{$e$};
        \filldraw[black] (1,1) circle (2pt) node[anchor=south east]{$f$};
        \filldraw[black] (2,0) circle (2pt) node[anchor=south west]{$g$};
    \end{tikzpicture}
    \end{center}
    It is also required that $\partial^2 f = 0$. The only way in which this can be achieved while keeping the picture horizontally and vertically simplified is if there is a vertical arrow from $f$ to $h$ of length $1$ and a diagonal arrow from $h$ to $d$. Since there are an even number of ways of traveling from $e$ to $h$ in two steps, we must have a horizontal arrow from $g$ to $h$. The situation is as depicted.
    
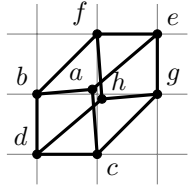
\begin{center}
    \begin{tikzpicture}[scale=0.8]
        \def\a{0.08}
        \draw[step=1.0,gray,thin] (-0.5, -1.5) grid (2.5,1.5);
        \draw[very thick] (0,0)--(1-\a,0+\a)--(1,-1)--(0,-1)--(0,0);
        \draw[very thick] (1-\a,0+\a) -- (2,1);
        \draw[very thick] (1+\a,0-\a) -- (0,-1);
        \draw[very thick] (0,0)--(1,1);
        \draw[very thick] (1,-1)--(2,0);
        \draw[very thick] (1,1)--(2,1)--(2,0)--(1+\a,0-\a)--(1,1);
        \filldraw[black] (1-\a,0+\a) circle (2pt) node[anchor=south east]{$a$};
        \filldraw[black] (0,0) circle (2pt) node[anchor=south east]{$b$};
        \filldraw[black] (1,-1) circle (2pt) node[anchor=north west]{$c$};
        \filldraw[black] (0,-1) circle (2pt) node[anchor=south east]{$d$};
        \filldraw[black] (2,1) circle (2pt) node[anchor=south west]{$e$};
        \filldraw[black] (1,1) circle (2pt) node[anchor=south east]{$f$};
        \filldraw[black] (2,0) circle (2pt) node[anchor=south west]{$g$};
        \filldraw[black] (1+\a,0-\a) circle (2pt) node[anchor=south west]{$h$};
    \end{tikzpicture}
    \end{center}
    Let us perform the basis change $g \mapsto g+Ua$ and $h \mapsto h+Ub$. This basis change removes all depicted diagonal arrows, keeps the basis horizontally and vertically simplified and does not add any new diagonal arrows adjacent to the original square with vertices $a, b, c, d$. To see this, note that $g$ and $h$ have no incoming diagonal arrows, since they are in the higher grading and $a$ and $b$ have no outgoing diagonal arrows since they are in the lower grading. 
    \item There is an incoming arrow at $b$.

    Let $e$ be the other endpoint of this diagonal arrow adjacent to $b$. Since $\partial^2 e = 0$, there is another generator $f$ with a vertical arrow from $e$ to $f$ of length $1$ and a diagonal arrow from $f$ to $d$ as depicted.
    
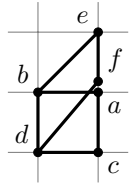
\begin{center}
    \begin{tikzpicture}[scale=0.8]
        \def\a{0.18}
        \draw[step=1.0,gray,thin] (-0.5, -1.5) grid (1.5,1.5);
        \draw[very thick] (0,0)--(1,0)--(1,-1)--(0,-1)--(0,0);
        \draw[very thick] (1,0+\a)--(0,-1);
        \draw[very thick] (0,0)--(1,1);
        \draw[very thick] (1,1)--(1,0+\a);
        \filldraw[black] (1,0) circle (2pt) node[anchor=north west]{$a$};
        \filldraw[black] (0,0) circle (2pt) node[anchor=south east]{$b$};
        \filldraw[black] (1,-1) circle (2pt) node[anchor=north west]{$c$};
        \filldraw[black] (0,-1) circle (2pt) node[anchor=south east]{$d$};
        \filldraw[black] (1,1) circle (2pt) node[anchor=south east]{$e$};
        \filldraw[black] (1,0+\a) circle (2pt) node[anchor=south west]{$f$};
    \end{tikzpicture}
    \end{center}
    Considering the hypothetical horizontal arrow adjacent to $f$, the following two cases arise naturally.
    \begin{enumerate}
        \item There is no horizontal arrow incoming to $f$. Since $f$ is in the higher grading, this means that the arrow from $e$ to $f$ is the only incoming arrow to $f$. In this case, one can make the change of basis $f \mapsto f+Ub$. This removes the depicted diagonal arrows from $e$ to $b$, $f$ to $d$ and does not add any diagonal arrows adjacent to the original square with vertices $a, b, c, d$ for the same reason as above.
        \item There is a horizontal arrow incoming to $f$. Let $g$ be its other endpoint. Because the complex has width $2$, the arrow from $g$ to $f$ must be of length $1$. Since $\partial^2 g = 0$, there is another way of traveling from $g$ to $d$ in two steps and hence there is a diagonal arrow from $g$ to $c$ as depicted.
        
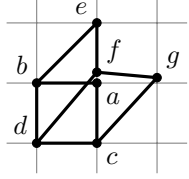
\begin{center}
            \begin{tikzpicture}[scale=0.8]
            \def\a{0.18}
            \draw[step=1.0,gray,thin] (-0.5, -1.5) grid (2.5,1.5);
            \draw[very thick] (0,0)--(1,0)--(1,-1)--(0,-1)--(0,0);
            \draw[very thick] (1,0+\a)--(0,-1);
            \draw[very thick] (0,0)--(1,1);
            \draw[very thick] (1,1)--(1,0+\a);
            \draw[very thick] (1,0+\a)--(2,0+\a/2);
            \draw[very thick] (1,-1)--(2,0+\a/2);
            \filldraw[black] (1,0) circle (2pt) node[anchor=north west]{$a$};
            \filldraw[black] (0,0) circle (2pt) node[anchor=south east]{$b$};
            \filldraw[black] (1,-1) circle (2pt) node[anchor=north west]{$c$};
            \filldraw[black] (0,-1) circle (2pt) node[anchor=south east]{$d$};
            \filldraw[black] (1,1) circle (2pt) node[anchor=south east]{$e$};
            \filldraw[black] (1,0+\a) circle (2pt) node[anchor=south west]{$f$};
            \filldraw[black] (2,0+\a/2) circle (2pt) node[anchor=south west]{$g$};
            \end{tikzpicture}
        \end{center}
        Now if there is a vertical arrow from $h$ to $g$, then it must be of length $1$ as well and the condition that $\partial^2 h = 0$ implies that there is a diagonal arrow from $h$ to $a$, moving us back to case (1). We may thus assume that there are no incoming vertical arrows adjacent to $g$ and so $g$ has no incoming arrows at all. This lets us perform a change of basis $g \mapsto g+Ua$ and $f \mapsto f+Ub$, which removes the diagonal arrows from the picture and does not add any diagonal arrows adjacent to the original square with vertices $a, b, c, d$.
    \end{enumerate}

    \item There is an incoming arrow at $c$.

    This case is completely analogous to the case with an incoming arrow at $b$ due to the symmetry.
    \item There is an incoming arrow at $d$.

    Let $e$ be the other endpoint of this diagonal arrow. If there is an incoming horizontal or vertical arrow at $e$, it must be of length $1$ since the complex has width $2$. A vertical arrow implies the existence of a diagonal arrow incoming to $b$ and a horizontal arrow implies the existence of a diagonal arrow incoming to $c$. This means that they can be dealt with first, since we are in case (2) or (3). Therefore, we can assume that there are no incoming arrows at $e$. Performing the change of basis $e \mapsto e + Ub$ hence removes the diagonal arrow adjacent to $d$ while it does not add any diagonal arrows adjacent to the original square with vertices $a, b, c, d$.
\end{enumerate}
In all cases, we have strictly decreased the number of diagonal arrows adjacent to the original $K_1$. Therefore, the process terminates after finitely many steps, culminating in the basis in which $K_1$ is disconnected from the rest of the picture, \emph{i.e.} $C \cong K_1 \oplus C'$ over $\F[U,V]$. 
\end{proof}
By Lemma \ref{lemma:remove K_1} we may assume in further analysis that there are no $K_1$'s in the decomposition -- if there are, we can disconnect them first. We now prove an analogous result for $K_0$'s.
\begin{lemma}\label{lemma:remove K_0}
If $C \cong K_0 \oplus C'$ over $\mathcal{R}$, then $C \cong K_0 \oplus C'$ over $\F[U,V]$.  
\end{lemma}
\begin{proof}
Let the generators of the special shape $K_0$ be labelled as in Figure \ref{fig:special shape}. By symmetry, it is again sufficient to deal with the three cases of potential incoming diagonal arrows to $K_0$. In the argument, we induct on the set of diagonal arrows adjacent to any of the vertices $a, b, c, d, e, f$ of $K_0$. In other words, whatever case we find ourselves in, we will perform a change of basis such that the resulting set of diagonal arrows adjacent to $K_0$ is a strict subset of the initial set of diagonal arrows adjacent to $K_0$.
\begin{enumerate}
    \item There is an incoming external diagonal arrow at $a$.

    Let $g$ be the other endpoint of this diagonal arrow adjacent to $a$. Since $\partial ^2 g = 0$, there is another way of traveling from $g$ to $f$ in two steps, which implies the existence of a generator $h$ with a vertical arrow from $g$ to $h$ of length $1$ and a diagonal arrow from $h$ to $f$.
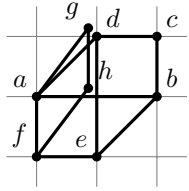
    \begin{center}
    \begin{tikzpicture}[scale = 0.8]
            \def\a{0.14}
            \draw[step=1.0,gray,thin] (0.5,-0.5) grid (3.5,2.5);
            \draw[black, very thick] (1,1) -- (3,1) -- (3,2) -- (2,2) -- (2,0) --(1,0) -- (1,1);
            \draw[black, very thick] (2,2) -- (1,1);
            \draw[black, very thick] (3,1) -- (2,0);
            \draw[black, very thick] (1,0) -- (2-\a,1+\a) -- (2-\a,2+\a) -- (1,1);
            \filldraw[black] (1,1) circle (2pt) node[anchor=south east]{$a$};
            \filldraw[black] (3,1) circle (2pt) node[anchor=south west]{$b$};
            \filldraw[black] (3,2) circle (2pt) node[anchor=south west]{$c$};
            \filldraw[black] (2,2) circle (2pt) node[anchor=south west]{$d$};
            \filldraw[black] (2,0) circle (2pt) node[anchor=south east]{$e$};
            \filldraw[black] (1,0) circle (2pt) node[anchor=south east]{$f$};
            \filldraw[black] (2-\a,2+\a) circle (2pt) node[anchor=south east]{$g$};
            \filldraw[black] (2-\a,1+\a) circle (2pt) node[anchor=south west]{$h$};
    \end{tikzpicture}
    \end{center}
    If $h$ has no incoming horizontal arrows, then $h$ has no incoming arrows at all and a change of basis $h \mapsto h + Ua$ can be used to remove the two diagonal arrows presently in the picture. This does not add any new diagonal arrows either. In the other case $h$ has an incoming horizontal arrow, necessarily of length $1$. Let $i$ be its other endpoint and note that $\partial^2 i = 0$ necessitates the existence of a diagonal arrow from $i$ to $e$. Note that $i$ has no incoming arrows - a vertical incoming arrow to $i$ would create a $K_1$ summand, but those have been ruled out by Lemma \ref{lemma:remove K_1}. Therefore one can perform the change of basis $h \mapsto h+Ua$ and $i \mapsto i+b$, which removes the diagonal arrow from $g$ to $a$ without adding any new diagonal arrows adjacent to any of the vertices $a, b, c, d, e, f$ of the original $K_0$.
    \item There is an incoming external diagonal arrow at $e$.

    Due to the symmetry, this case is completely analogous to the case with an incoming external diagonal arrow at $a$.
    \item There is an incoming diagonal arrow at $f$.

    Let $g$ be the other endpoint of the diagonal arrow adjacent to $f$. If $g$ has any incoming horizontal or vertical arrows, they must be of length $1$ and the condition $\partial^2 = 0$ implies the existence of new diagonal arrows adjacent to $a$ or $e$, moving us back to the previous cases. Therefore, we may assume that $g$ has no incoming arrows. Performing the basis change $g \mapsto g + Ua$ removes the starting diagonal arrow without adding any new diagonal arrows adjacent to any of the vertices $a, b, c, d, e, f$ of the original $K_0$.
\end{enumerate}
In all cases, the original diagonal arrow was removed after performing the change of basis. While it is possible that some diagonal arrows were introduced anew in the process (for example, from $i+b$ in Case (1)), no new diagonal arrows adjacent to $K_0$ were added. Therefore, the process terminates after finitely many steps, resulting in the basis in which $K_0$ is disconnected from the rest of the picture, \emph{i.e.} $C \cong K_0 \oplus C'$ over $\F[U,V]$. 
\end{proof}
Finally, the external diagonal arrows adjacent to the remaining local systems are removed.
\begin{lemma}\label{lemma:remove the rest}
Let $L$ be any other local system of width $2$, \emph{i.e.} a $K_n$ or $\K{n}$ for some $n \geq 2$. If $C \cong L \oplus C'$ over $\mathcal{R}$, then $C \cong L \oplus C'$ over $\F[U,V]$.
\end{lemma}
\begin{proof}
We follow the same strategy as in the previous lemma. At each step, we will strictly decrease the number of external diagonal arrows adjacent to $L$. Note that the analysis of the case in which $K_n$ has outgoing diagonal arrows is completely analogous to the analysis of the case in which $\K{n}$ has incoming diagonal arrows. As such, it is sufficient to show how to remove all incoming diagonal arrows into $L$. This splits the analysis into $4$ cases.
\begin{enumerate}
    \item $L = K_n$ for some odd $n$.

    There are $n-2$ generators of $L$ that are in the lower grading. Split them into
    \begin{itemize}[label=$-$]
        \item \textcolor{blue}{blue} generators, all of whose adjacent horizontal and vertical arrows are incoming, and
        \item \textcolor{dark green}{green} generators, all of whose adjacent horizontal and vertical arrows are outgoing.
    \end{itemize}
    Note that the generators on the ``lower edge'' of $L$ are alternating between blue and green. If there is an incoming external diagonal arrow into a green generator, the condition $\partial^2=0$ implies the existence of an incoming external diagonal arrow into a blue generator. Hence, it is sufficient to show how to remove all external diagonal arrows adjacent to blue generators.

    Let $a \in L$ be a blue generator adjacent to some external diagonal arrow as in Figure \ref{fig:L middle}. Let $x$ be the other endpoint of this arrow. If $x$ has no incoming arrows, then the basis change $x \mapsto x + b$ removes the external diagonal arrow from $x$ to $a$. Note that, whenever $b$ has other outgoing diagonal arrows to $y_1, \dots, y_n$, the above change of basis creates new diagonal arrows from $x$ to $y_1, \dots, y_n$. However, $x \notin L$ and $y_1, \dots, y_n \notin L$ so the newly created diagonal arrows are not adjacent to $L$. 
    
    Similarly, let us now assume that $x$ has some horizontal incoming arrow. Let its other endpoint be $y$ as in Figure \ref{fig:L middle 2}. Since $\partial^2y=0$, there must be a diagonal arrow from $y$ to $c$, which in turn implies the existence of a vertical arrow from $y$ to $z$ of length $1$ and a diagonal arrow from $z$ to $e$. If now $x$ and $z$ are not adjacent to any undrawn incoming arrows, the basis change $x \mapsto x+b$, $y\mapsto y+d$, $z\mapsto z+f$ can be used to remove the diagonal arrows from $x$ to $a$, $y$ to $c$ and $z$ to $e$ without adding any new diagonal arrows adjacent to the vertices of $L$.
    \begin{figure}
\centering
\begin{subfigure}{0.45\textwidth}
    $$
    \begin{tikzpicture}[scale = 0.8]
            \def\a{0.15}
            \draw[step=1.0,gray,thin] (-2.5, 0.5) grid (3.5,6.5);
            \draw[black, very thick] (-1,3) -- (0,3) -- (0,2) -- (1,2) -- (1,1) -- (3,1) -- (3,2) -- (2,2) -- (2,3) -- (1,3) -- (1,4) -- (0,4) -- (0,5) -- (-1, 5) -- (-1, 6) -- (-2, 6) -- (-2, 4) --(-1,4) --(-1,3);
            \draw[black, very thick] (2,2) -- (1,1);
            \draw[black, very thick] (2,3) -- (1,2);
            \draw[black, very thick] (1,3) -- (0,2);
            \draw[black, very thick] (1,4) -- (0,3);
            \draw[black, very thick] (0,4) -- (-1,3);
            \draw[black, very thick] (-2,4) -- (-1,5);
            \draw[black, very thick] (-1,4) -- (0,5);
            \draw[black, very thick] (-1, 3) to[out=30, in=270] (0+\a, 4+\a);
            \filldraw[black] (-1,6) circle (2pt) node[anchor=south east]{};
            \filldraw[black] (-2,6) circle (2pt) node[anchor=south east]{};
            \filldraw[blue] (-2,4) circle (2pt) node[anchor=south east]{};
            \filldraw[dark green] (-1,4) circle (2pt) node[anchor=south east]{};
            \filldraw[black] (-1,5) circle (2pt) node[anchor=south east]{};
            \filldraw[blue] (-1,3) circle (2pt) node[anchor=north east]{$a$};
            \filldraw[dark green] (0,3) circle (2pt) node[anchor=north east]{};
            \filldraw[blue] (0,2) circle (2pt) node[anchor=north east]{};
            \filldraw[dark green] (1,2) circle (2pt) node[anchor=south east]{};
            \filldraw[blue] (1,1) circle (2pt) node[anchor=south east]{};
            \filldraw[black] (3,1) circle (2pt) node[anchor=south east]{};
            \filldraw[black] (3,2) circle (2pt) node[anchor=south east]{};
            \filldraw[black] (2,2) circle (2pt) node[anchor=south east]{};
            \filldraw[black] (2,3) circle (2pt) node[anchor=south east]{};
            \filldraw[black] (1,3) circle (2pt) node[anchor=south east]{};
            \filldraw[black] (1,4) circle (2pt) node[anchor=south east]{};
            \filldraw[black] (0,4) circle (2pt) node[anchor=south east]{$b$};
            \filldraw[black] (0,5) circle (2pt) node[anchor=south east]{};
            \filldraw[black] (0+\a,4+\a) circle (2pt) node[anchor=south west]{$x$};
    \end{tikzpicture}
    $$
    \caption{} 
    \label{fig:L middle}
\end{subfigure}
\begin{subfigure}{0.45\textwidth}
    $$
    \begin{tikzpicture}[scale = 0.8]
            \def\a{0.15}
            \draw[step=1.0,gray,thin] (-2.5, 0.5) grid (3.5,6.5);
            \draw[black, very thick] (-1,3) -- (0,3) -- (0,2) -- (1,2) -- (1,1) -- (3,1) -- (3,2) -- (2,2) -- (2,3) -- (1,3) -- (1,4) -- (0,4) -- (0,5) -- (-1, 5) -- (-1, 6) -- (-2, 6) -- (-2, 4) --(-1,4) --(-1,3);
            \draw[black, very thick] (2,2) -- (1,1);
            \draw[black, very thick] (2,3) -- (1,2);
            \draw[black, very thick] (1,3) -- (0,2);
            \draw[black, very thick] (1,4) -- (0,3);
            \draw[black, very thick] (0,4) -- (-1,3);
            \draw[black, very thick] (-2,4) -- (-1,5);
            \draw[black, very thick] (-1,4) -- (0,5);
            \draw[black, very thick] (-1, 3) to[out=30, in=270] (0+\a, 4+\a);
            \draw[black, very thick] (0, 3) to[out=30, in=270] (1+\a, 4+\a);
            \draw[black, very thick] (0, 2) to[out=30, in=270] (1+\a, 3+\a);
            \draw[black, very thick] (0+\a,4+\a) -- (1+\a,4+\a) -- (1+\a,3+\a);
            \filldraw[black] (-1,6) circle (2pt) node[anchor=south east]{};
            \filldraw[black] (-2,6) circle (2pt) node[anchor=south east]{};
            \filldraw[blue] (-2,4) circle (2pt) node[anchor=south east]{};
            \filldraw[dark green] (-1,4) circle (2pt) node[anchor=south east]{};
            \filldraw[black] (-1,5) circle (2pt) node[anchor=south east]{};
            \filldraw[blue] (-1,3) circle (2pt) node[anchor=north east]{$a$};
            \filldraw[dark green] (0,3) circle (2pt) node[anchor=north east]{$c$};
            \filldraw[blue] (0,2) circle (2pt) node[anchor=north east]{$e$};
            \filldraw[dark green] (1,2) circle (2pt) node[anchor=south east]{};
            \filldraw[blue] (1,1) circle (2pt) node[anchor=south east]{};
            \filldraw[black] (3,1) circle (2pt) node[anchor=south east]{};
            \filldraw[black] (3,2) circle (2pt) node[anchor=south east]{};
            \filldraw[black] (2,2) circle (2pt) node[anchor=south east]{};
            \filldraw[black] (2,3) circle (2pt) node[anchor=south east]{};
            \filldraw[black] (1,3) circle (2pt) node[anchor=east]{$f$};
            \filldraw[black] (1,4) circle (2pt) node[anchor=south east]{$d$};
            \filldraw[black] (0,4) circle (2pt) node[anchor=south east]{$b$};
            \filldraw[black] (0,5) circle (2pt) node[anchor=south east]{};
            \filldraw[black] (0+\a,4+\a) circle (2pt) node[anchor=south west]{$x$};
            \filldraw[black] (1+\a,4+\a) circle (2pt) node[anchor=south west]{$y$};
            \filldraw[black] (1+\a,3+\a) circle (2pt) node[anchor=south west]{$z$};
    \end{tikzpicture}
    $$
    \caption{} 
    \label{fig:L middle 2}
\end{subfigure}
    \caption{Subfigure (\textsc{a}) demonstrates the case (1) of Lemma \ref{lemma:remove the rest}. If $x$ has no incoming arrows, then the simple basis change $x \mapsto x + b$ removes the diagonal arrow from $x$ to $a$ without adding any new arrows adjacent to the vertices of $L$. If $x$ has an incoming horizontal arrow, then the requirement that $\partial^2=0$ forces the configuration depicted in Subfigure (\textsc{b}). If $z$ has no incoming horizontal arrows, then the basis change $x \mapsto x+b$, $y\mapsto y+d$, $z\mapsto z+f$ removes the diagonal arrows from $x$ to $a$, $y$ to $c$, and $z$ to $e$ without adding any new diagonal arrows adjacent to the vertices of $L$.}
    \label{fig:L}
    \end{figure}
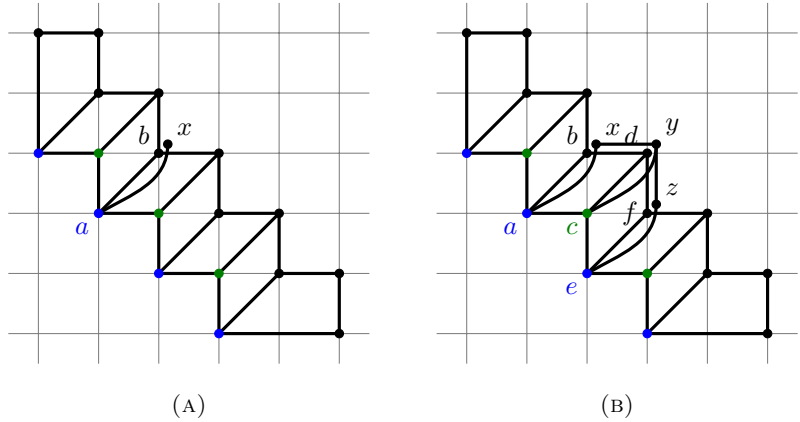
    Similar reasoning works more generally if $x$ or $z$ are adjacent to some additional horizontal or vertical arrows: one keeps drawing the upper width $1$ staircases until its endpoints are not adjacent to any undrawn incoming arrows. At this point, the natural change of basis removes the initial diagonal arrow without adding new diagonal arrows adjacent to $L$.

    The above discussion implicitly assumes that the endpoints of the width $1$ staircase will stop being adjacent to additional horizontal and vertical arrows before we ``reach the end of $L$''. We now argue that this is indeed always the case. Let $a$ be the lowest blue generator and let $b$ be the other endpoint of the external arrow adjacent to $a$. For the sake of contradiction, assume that $b$ is adjacent to an incoming horizontal arrow and denote its other endpoint by $c$ as depicted.
    
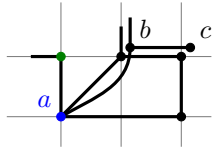
\begin{center}
    \begin{tikzpicture}[scale=0.8]
    \def\a{0.15}
    \draw[step=1.0,gray,thin] (-0.90, -0.5) grid (2.5,1.9);
    \draw[very thick] (0,0)--(2,0)--(2,1)--(1,1)--(0,0) -- (0,1);
    \draw[very thick] (0,1) -- (-0.5, 1);
    \draw[very thick] (1,1) -- (1, 1.5);
    \draw[very thick] (1+\a,1+\a) -- (1+\a, 1.5+\a);
    \draw[very thick] (1+\a,1+\a) -- (2+\a, 1+\a);
    \draw[very thick] (0,0) to[out=30, in=270] (1+\a, 1+\a);
    \filldraw[black] (2,0) circle (2pt) node[anchor=south west]{};
    \filldraw[dark green] (0,1) circle (2pt) node[anchor=south west]{};
    \filldraw[blue] (0,0) circle (2pt) node[anchor=south east]{$a$};
    \filldraw[black] (1,1) circle (2pt) node[anchor=south east]{};
    \filldraw[black] (1+\a,1+\a) circle (2pt) node[anchor=south west]{$b$};
    \filldraw[black] (2+\a,1+\a) circle (2pt) node[anchor=south west]{$c$};
    \filldraw[black] (2,1) circle (2pt) node[anchor=south east]{};
    \end{tikzpicture}
    \end{center}
    Now $\partial^2 c=0$ so there must be another way of traveling from $c$ to $a$ in two steps. However, this is not possible while maintaining the complex vertically and horizontally simplified. This is a contradiction and means that $b$ was not adjacent to any undrawn incoming arrows as required. The other end of $L$ lends itself to a completely analogous treatment.
\end{enumerate}
The analysis of the remaining cases is largely similar. If the endpoints of the width $1$ staircase stop being adjacent to additional horizontal and vertical arrows before one reaches the end of $L$, we perform the same change of basis as in case (1). In the remaining cases, we therefore focus solely on the situation when the width $1$ staircase reaches the end of $L$. As we shall see, there are only $2$ types of ends of $L$, so $(2)$ is the only real outstanding case. 
\begin{enumerate}[resume]
    \item $L = \K{n}$ for some odd $n$.

    Let $a$ be the lowest blue generator and let $b$ be the other endpoint of the external arrow adjacent to $a$. If $b$ is adjacent to an incoming horizontal arrow, let $c$ denote its other endpoint as depicted.
    
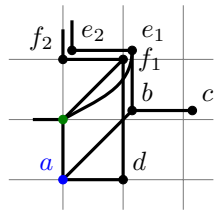
\begin{center}
    \begin{tikzpicture}[scale=0.8]
    \def\a{0.15}
    \draw[step=1.0,gray,thin] (-0.90, -1.5) grid (2.5,1.9);
    \draw[very thick] (0,0)--(0,-1)--(1,-1)--(1,1)--(0,1);
    \draw[very thick] (0,0) -- (-0.5, 0);
    \draw[very thick] (0,1) -- (0, 1.5);
    \draw[very thick] (0+\a,1+\a) -- (0+\a, 1.5+\a);
    \draw[very thick] (0,0) -- (1,1);
    \draw[very thick] (0,-1) -- (1+\a,0+\a)--(2+\a,0+\a);
    \draw[very thick] (0+\a, 1+\a) -- (1+\a,1+\a) -- (1+\a, 0+\a);
    \draw[very thick] (0,0) to[out=30, in=270] (1+\a, 1+\a);
    \filldraw[] (1,1) circle (2pt);
    \node[] at (1+3*\a,1) {$f_1$};
    \filldraw[blue] (0,-1) circle (2pt) node[anchor=south east]{$a$};
    \filldraw[] (1,-1) circle (2pt) node[anchor=south west]{$d$};
    \filldraw[] (1+\a,0+\a) circle (2pt) node[anchor=south west]{$b$};
    \filldraw[] (2+\a,0+\a) circle (2pt) node[anchor=south west]{$c$};
    \filldraw[] (1+\a,1+\a) circle (2pt) node[anchor=south west]{$e_1$};
    \filldraw[] (0+\a,1+\a) circle (2pt) node[anchor=south west]{$e_2$};
    \filldraw[] (0,1) circle (2pt) node[anchor=south east]{$f_2$};
    \filldraw[dark green] (0,0) circle (2pt) node[anchor=south east]{};
    \end{tikzpicture}
    \end{center}
    Since $\partial^2 c = 0$, there must be a diagonal arrow from $c$ to $d$. Note that we do not have control over what happens at $c$ -- if might or might not be adjacent to any vertical arrows, but this will fortunately not be important for us. We perform the change of basis $b \mapsto b+ Vd$ and $e_i \mapsto e_i + f_i$ where $e_i$ are the remaining generators of the width $1$ staircase and $f_i$ are the generators of $L$ drawn in the same spatial location in the plane. Note that the other end of $L$ allows for a very similar analysis so we omit it here.
    
    \item $L = K_n$ for some even $n$.

    The upper end of $L$ locally looks like the ends of $L$ in case (1) and the lower end of $L$ locally looks like the ends of $L$ in case (2). As a result, the same arguments works.
    \item $L = \K{n}$ for some even $n$.

    The upper end of $L$ locally looks like the ends of $L$ in case (2) and the lower end of $L$ locally looks like the ends of $L$ in case (1). As a result, the same arguments works.
\end{enumerate}
In all cases, the number of external diagonal arrows adjacent to $L$ has strictly decreased after the change of basis was performed. In fact, no diagonal arrows adjacent to vertices of $L$ were drawn anew. This guarantees that the process is finite and results in a basis for $C$ such that $L$ is a direct summand.
\end{proof}
We are now ready to prove the main result of this paper.
\begin{proof}[Proof of Theorem \ref{thm:width2}]
    Proposition \ref{prop:width2 over R} establishes the corresponding classification result over the ring $\mathcal{R} = \frac{\F[U,V]}{(UV)}$. Proposition \ref{prop:disconnect everything} shows that the direct summands as chain complexes over $\mathcal{R}$ remain the direct summands as chain complexes over $\F[U,V]$. The unique remaining direct summand is the standard complex.
\end{proof}

\subsection{Knots of higher width}
As established in the preceding subsections, the classification of chain homotopy equivalence types of chain complexes over $\F[U,V]$ becomes much more complex as the number of diagonals is increased from $1$ to $2$. On $3$ diagonals, Observation \ref{obs:enough to do single local systems} no longer applies and many new, qualitatively different examples of local systems that can be lifted to complexes over $\F[U,V]$ emerge, such as the complexes depicted in Figure \ref{fig:width 3}.
\begin{figure}
    \begin{subfigure}[b]{0.40\textwidth}
        \centering
    $$
        \begin{tikzpicture}[scale = 1.0]
            \draw[step=1.0,gray,thin] (-0.5,0.5) grid (3.5,4.5);
            \draw[black, very thick] (1,1) -- (2,1) -- (2,4) -- (3,4) -- (3,3) -- (1,3) -- (1,4) -- (0,4) -- (0,2) -- (1,2) -- (1,1);
            \draw[black, very thick] (1,3) -- (0,2);
            \draw[black, very thick] (1,3) -- (2,4);
            \draw[black, very thick] (2,1) -- (3,3);
            \draw[black, very thick] (1,2) -- (3,3);
            \draw[black, very thick] (2,4) -- (1,2);
            \filldraw[black] (1,1) circle (2pt) node[anchor=south east]{};
            \filldraw[black] (2,1) circle (2pt) node[anchor=south east]{};
            \filldraw[black] (2,4) circle (2pt) node[anchor=south east]{};
            \filldraw[black] (3,4) circle (2pt) node[anchor=south east]{};
            \filldraw[black] (3,3) circle (2pt) node[anchor=south east]{};
            \filldraw[black] (1,3) circle (2pt) node[anchor=south east]{};
            \filldraw[black] (1,4) circle (2pt) node[anchor=south east]{};
            \filldraw[black] (0,4) circle (2pt) node[anchor=south east]{};
            \filldraw[black] (0,2) circle (2pt) node[anchor=south east]{};
            \filldraw[black] (1,2) circle (2pt) node[anchor=south east]{};
            \filldraw[black] (0,1) circle (2pt) node[anchor=south east]{};
        \end{tikzpicture}
    $$
    \end{subfigure}
    \begin{subfigure}[b]{0.40\textwidth}
        \centering
        \begin{tikzpicture}[scale=0.9]
            \draw[step=1.0,gray,thin] (0.5,0.5) grid (3.5,3.5);
            \def\a{0.1}
            \draw[black, very thick] (1+\a, 1+\a) -- (1+\a, 3-\a) -- (3-\a, 3-\a) -- (3-\a,1+\a);
            \draw[black, very thick] (1-\a, 1-\a) -- (1-\a, 3+\a) -- (3+\a, 3+\a) -- (3+\a,1-\a) -- (1+\a, 1+\a);
            \draw[black, very thick] (3-\a, 1+\a) -- (1-\a, 1-\a);

            \draw[black, very thick] (2,2) to[out=-100, in=10] (1-\a, 1-\a);
            \draw[black, very thick] (3-\a,3-\a) to[out=-100, in=10] (2,2);

            \draw[black, very thick] (2,2) to[out=190, in=80] (1+\a, 1+\a);
            \draw[black, very thick] (3+\a,3+\a) to[out=190, in=80] (2,2);
            
            \filldraw[black] (1+\a,1+\a) circle (2pt) node[anchor=south west]{};
            \filldraw[black] (1+\a,3-\a) circle (2pt) node[anchor=north west]{};
            \filldraw[black] (3-\a,1+\a) circle (2pt) node[anchor=south east]{};
            \filldraw[black] (3-\a,3-\a) circle (2pt) node[anchor=north east]{};
            \filldraw[black] (1-\a,1-\a) circle (2pt) node[anchor=north east]{};
            \filldraw[black] (1-\a,3+\a) circle (2pt) node[anchor=south east]{};
            \filldraw[black] (3+\a,1-\a) circle (2pt) node[anchor=north west]{};
            \filldraw[black] (3+\a,3+\a) circle (2pt) node[anchor=south west]{};
            \filldraw[black] (2,2) circle (2pt) node[anchor=south east]{};
        \end{tikzpicture}
    \end{subfigure}
    \caption{Chain complexes over $\F[U,V]$ of width $3$ whose mod $UV$ reduction is a direct sum of $\F[U,V]$ and a local system.}
    \label{fig:width 3}
\end{figure}
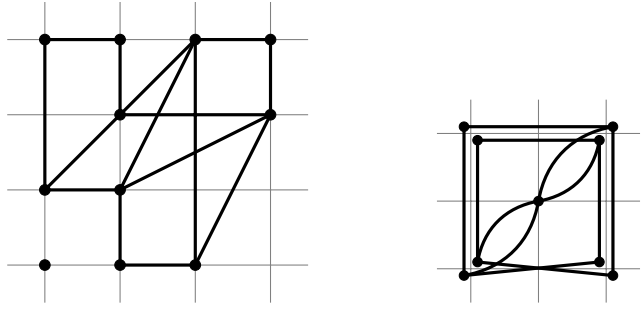
We believe the class of chain homotopy types of complexes on $3$ diagonals is already too diverse to admit a concise description. However, the following natural question might be within reach.
\begin{question}\label{question:nontrivial local system on 3 diagonals}
Is there a knot Floer-like complex of width $3$ that has a nontrivial local system?
\end{question}
Theorem \ref{thm:width2} shows that knots of width $2$ have no nontrivial local systems and Example $P$ in \cite[Theorem 1.4]{popovic2023link} shows that there is a knot Floer-like complex of width $4$ with a nontrivial local system. The answer to Question \ref{question:nontrivial local system on 3 diagonals} would determine the exact complexity needed for the existence of nontrivial local systems in the algebraic setting.

\section{Realization of local systems}\label{sec:realization}
This section shows that, with the possible exception of the \emph{special shape}, all local systems from Theorem \ref{thm:width2} appear as direct summands of $\CFK_{\F[U,V]}(K)$ for some knots $K$. We will use certain iterated cables of the figure eight knot as $K$, although we later remark that other options are possible. The main technology of this section are the Hanselman-Rasmussen-Watson's interpretation of $\CFK_{\mathcal{R}}(K)$ as immersed curves on a punctured torus \cite{hanselman2023bordered, hanselman2022heegaard} and Hanselman-Watson's cabling results \cite{hanselman2023cabling}. Let us review the results of these papers.

We begin with an explanation about how knot Floer complexes $\CFK_{\mathcal{R}}(K)$ can be represented by immersed curves. As is standard in \cite{hanselman2022heegaard}, we draw the immersed curves on an infinite cylinder $[0,1] \times \R$ with $\Z$ many punctures at $\{\frac{1}{2}\} \times \{\frac{1}{2}+k \ | \ k\in\Z\}$. This is a $\Z$-sheeted covering space of a punctured torus, so the composition with the covering projection can be used to recover the curves on a torus.

The knot Floer complex generators are drawn in the middle of the strip between consecutive punctures. Their heights are determined by the Alexander grading $A = \frac{1}{2}(\gr_U-\gr_V)$ so that the generators with the same Alexander grading lie between the same punctures. Horizontal arrows correspond to the arcs connecting the generators on the right and vertical arrows correspond to the arcs connecting the generators on the left. It is readily verified that the above requirements imply that the arrows of length $l$ correspond to arcs passing by $l$ punctures until they switch to the other side. See Figure \ref{fig:immersed curve fig 8} for the $\CFK_{\mathcal{R}}(4_1)$ of the figure eight knot and the corresponding immersed curve. See also \cite[Section 6]{popovic2023link}.
\begin{figure}[t]
    \begin{subfigure}[b]{0.40\textwidth}
    \centering
    \begin{tikzpicture}[scale=0.9]
    \def\a{0.2}
    \draw[step=2,gray,thin] (-2.5,-0.5) grid (0.5,2.5);
    \draw[very thick] (0,0)--(-2,0)--(-2,2) -- (0,2) -- (0,0);

    \filldraw[] (0,0) circle (2pt) node[anchor=north west]{$c$};
    \filldraw[] (-2,0) circle (2pt) node[anchor=south east]{$d$};
    \filldraw[] (-2,2) circle (2pt) node[anchor=north west]{$b$};
    \filldraw[] (0,2) circle (2pt) node[anchor=north west]{$a$};
    \filldraw[] (-2-\a,0-\a) circle (2pt) node[anchor=north east]{$e$};    
    \end{tikzpicture}
    \caption{}
    \label{fig:CFK of figure 8 knot}
    \end{subfigure}
    \begin{subfigure}[b]{0.40\textwidth}
    \centering
    \begin{tikzpicture}[scale=0.9]
    \def\b{0.33}
    \draw[very thick] (0,0.5) -- (0,4.5);
    \draw[very thick] (2,0.5) -- (2,4.5);
    \draw[very thick] (0,2.5) -- (2,2.5);
    \filldraw[] (1,1) circle (2pt);
    \filldraw[] (1,2) circle (2pt);
    \filldraw[] (1,3) circle (2pt);
    \filldraw[] (1,4) circle (2pt);

    \draw[dashed] (0, 1.5) -- (2.75,1.5){};
    \draw[dashed] (0, 2.5) -- (3,2.5){};
    \draw[dashed] (0, 3.5) -- (3,3.5){};

    \draw[] (3,4.5) node[anchor = west]{$A$};
    \draw[] (2.75,1.5) node[anchor = west]{$-1$};
    \draw[] (3,2.5) node[anchor = west]{$0$};
    \draw[] (3,3.5) node[anchor = west]{$1$};

    \filldraw[] (1,1.5) circle (1.2pt){};
    \filldraw[] (1,2.7) circle (1.2pt){};
    \filldraw[] (1,3.5) circle (1.2pt){};
    \filldraw[] (1,2.3) circle (1.2pt){};
    \filldraw[] (1,2.5) circle (1.2pt){};

    \draw[] (1.2,1.35) node[]{$c$};
    \draw[] (0.8,3.65) node[]{$b$};
    \draw[] (1.3,2.3) node[]{$d$};
    \draw[] (0.85,2.8) node[]{$a$};
    \draw[] (1.2,2.65) node[]{$e$};

    \draw[very thick] (1,1.5) to[out=0, in=340] (1, 2.3);
    \draw[very thick] (1,2.3) to[out=160, in=180] (1, 3.5);
    
    \draw[very thick] (1,3.5) to[out=0, in=20] (1, 2.7);
    \draw[very thick] (1,2.7) to[out=200, in=180] (1, 1.5);
    \end{tikzpicture}
    \caption{}
    \end{subfigure}    
    \caption{A knot Floer complex $\CFK_{\mathcal{R}}(4_1)$ associated to the figure eight knot in $(\textsc{a})$ and the corresponding immersed curve $\gamma$ in $(\textsc{b})$. The figure also contains the information about the Alexander gradings of the depicted generators.}
    \label{fig:immersed curve fig 8}
\end{figure}
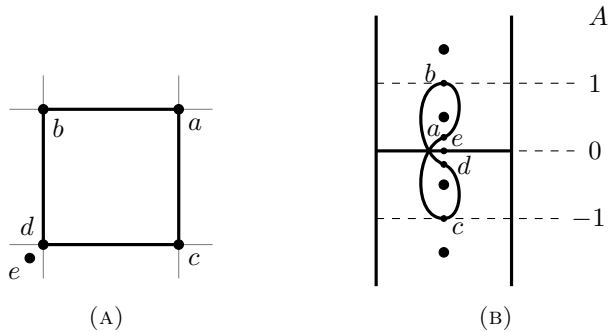

In \cite{hanselman2023cabling}, the authors present a very nice geometric way of working out the immersed curves of the $(p,q)$-cable of $K$ from the immersed curves of $K$.
\begin{enumerate}
    \item Draw $p$ copies of the immersed curves and the punctures associated to $\CFK_{\mathcal{R}}(K)$. The $i^{\text{th}}$ copy is drawn on $[i-1, i] \times \R$, it is stretched vertically by a factor of $p$ and then moved downwards by $qi$. That is, to draw the $i^{\text{th}}$ copy, we apply the transformation $(x, y) \mapsto (x+i, py-qi)$ to the original immersed curve associated to $\CFK_{\mathcal{R}}(K)$.
    \item Apply the diffeomorphism $[0, p] \times \R \to [0,1] \times \R$, which moves all punctures horizontally and sends them to $\{\frac{1}{2}\} \times \{\frac{1}{2}+k \ | \ k\in\Z\}$. The images of the immersed curves under this diffeomorphism are the immersed curves of $K_{p,q}$. See Figure \ref{fig:cabling} for an example.
\end{enumerate}
We will now see this in action as we construct cables of the figure eight knot whose $\CFK_{\F[U,V]}(K)$ contains width $2$ local systems. Let us remind ourselves that $K_0, K_1, \dots$ denote the trivial local systems as drawn in the statement of Theorem \ref{thm:width2} -- $K_0$ is the special shape, $K_1$ is a $1 \times 1$ square, $K_2$ is a $2 \times 1$ rectangle, and so on.
\begin{proposition}\label{prop:realization}
For all $n \geq 1$, $K_n$ is a direct summand of $\CFK_{\F[U,V]}(K)$ for some $K$.
\begin{itemize}[label=$-$]
    \item If $n=1$, one can take $K=4_1$.
    \item If $n=2$, one can take $K=(4_1)_{2,1}$.
    \item If $n \geq 3$ is odd, one can take $K = (4_1)_{n,-2}$.
    \item If $n \geq 4$ is even, one can take $K = ((4_1)_{\frac{n}{2},1})_{2,1}$.
\end{itemize}
\end{proposition}
In the statement of the proposition, $4_1$ denotes the figure eight knot and the subscript $_{p,q}$ denotes the $(p,q)$-cable of the underlying knot.
\begin{proof}
It is well-known that $\CFK_{\F[U,V]}(4_1)$ is the complex in Figure \ref{fig:CFK of figure 8 knot}, establishing the case $n=1$. In all other cases, rather than computing the immersed curves representation of $\CFK_{\mathcal{R}}(K)$ in its entirety, we will only be interested in some of its structure. In particular, we always disregard the standard complex summand and most of the local systems. 

Let first $n$ be even. Consider the $(\frac{n}{2},1)$-cable of the figure eight knot, whose immersed curve is drawn in Figure \ref{fig:immersed curve fig 8}. Applying the cabling process as described above, we draw $\frac{n}{2}$ copies of the figure eight curve staggered in height as depicted in Figure \ref{subfigure: 4 copies of figure 8}.
\begin{figure}[t]
\begin{subfigure}[b]{0.40\textwidth}
    \centering
    \begin{tikzpicture}[scale=0.9]
    \def\b{0.33}
    \draw[very thick] (0,-2.5) -- (0,5.5);
    \draw[very thick] (5,-2.5) -- (5,5.5);
    \filldraw[] (1,5) circle (2pt);
    \filldraw[] (1,1) circle (2pt);

    \filldraw[] (2,4) circle (2pt);
    \filldraw[] (2,0) circle (2pt);

    \filldraw[] (3,3) circle (2pt);
    \filldraw[] (3,-1) circle (2pt);

    \filldraw[] (4,2) circle (2pt);
    \filldraw[] (4,-2) circle (2pt);

    \draw[very thick] (1,5.5) to[out=0, in=50, looseness=0.8] (1, 3);
    \draw[very thick] (1,3) to[out=230, in=180, looseness=0.8] (1, 0.5);
    \draw[very thick] (1,3) to[out=120, in=180, looseness=0.8] (1, 5.5);
    \draw[very thick] (1,3) to[out=-60, in=0, looseness=0.8] (1, 0.5);

    \draw[very thick] (2,4.5) to[out=0, in=50, looseness=0.8] (2, 2);
    \draw[very thick] (2,2) to[out=230, in=180, looseness=0.8] (2, -0.5);
    \draw[very thick] (2,2) to[out=120, in=180, looseness=0.8] (2, 4.5);
    \draw[very thick] (2,2) to[out=-60, in=0, looseness=0.8] (2, -0.5);

    \draw[very thick] (3,3.5) to[out=0, in=50, looseness=0.8] (3, 1);
    \draw[very thick] (3,1) to[out=230, in=180, looseness=0.8] (3, -1.5);
    \draw[very thick] (3,1) to[out=120, in=180, looseness=0.8] (3, 3.5);
    \draw[very thick] (3,1) to[out=-60, in=0, looseness=0.8] (3, -1.5);

    \draw[very thick] (4,2.5) to[out=0, in=50, looseness=0.8] (4,0);
    \draw[very thick] (4,0) to[out=230, in=180, looseness=0.8] (4, -2.5);
    \draw[very thick] (4,0) to[out=120, in=180, looseness=0.8] (4, 2.5);
    \draw[very thick] (4,0) to[out=-60, in=0, looseness=0.8] (4, -2.5);
    \end{tikzpicture}
    \caption{}
    \label{subfigure: 4 copies of figure 8}
    \end{subfigure}    
    \begin{subfigure}[b]{0.40\textwidth}
    \centering
    \begin{tikzpicture}[scale=0.9]
    \def\b{0.33}
    \draw[very thick] (0,-2.5) -- (0,5.5);
    \draw[very thick] (2,-2.5) -- (2,5.5);
    \filldraw[] (1,5) circle (2pt);
    \filldraw[] (1,4) circle (2pt);
    \filldraw[] (1,3) circle (2pt);
    \filldraw[] (1,2) circle (2pt);
    \filldraw[] (1,1) circle (2pt);
    \filldraw[] (1,0) circle (2pt);
    \filldraw[] (1,-1) circle (2pt);
    \filldraw[] (1,-2) circle (2pt);

    \draw[very thick] (1,2.5) to[out=180, in=120, looseness=1] (1, 1.5);
    \draw[very thick] (1,1.5) to[out=-60, in=0, looseness=0.8] (1, -2.5);
    \draw[very thick] (1,-1.5) to[out=60, in=0, looseness=0.8] (1, 2.5);
    \draw[very thick] (1,-1.5) to[out=240, in=180, looseness=0.8] (1, -2.5);
    \end{tikzpicture}
    \caption{}
    \label{subfigure:result of cabling}
    \end{subfigure}    
    \caption{Calculation of the immersed curve of the $(\frac{n}{2},1)$-cable of the figure eight knot $4_1$ in case $n=8$. Figure (\textsc{a}) shows $\frac{n}{2}$ copies, staggered in height, of the immersed curve associated to $4_1$. Figure (\textsc{b}) shows what happens to the rightmost immersed curve after translating the punctures so that they lie on the same vertical line. In terms of local systems, this corresponds to a horizontal $\frac{n}{2} \times 1$ rectangle.}
    \label{fig:cabling}
\end{figure}
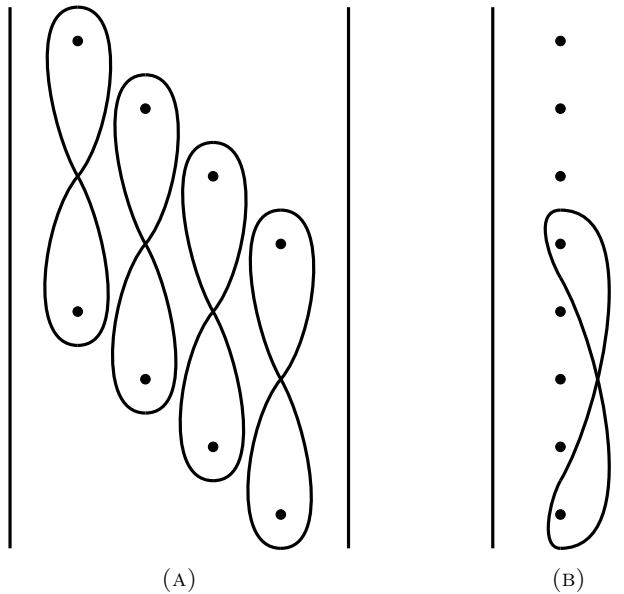
This is the step (1) of the cabling process. In step (2), we translate the punctures horizontally, which modifies the curves. See Figure \ref{subfigure:result of cabling} for the result of this operation on the rightmost curve. We convince ourselves that it corresponds to the local system with the shape of a horizontal $\frac{n}{2} \times 1$ rectangle. Similarly, the leftmost curve corresponds to the vertical $\frac{n}{2} \times 1$ rectangle. Note that in particular this shows that $2 \times 1$ rectangles $K_2$ appear as summands of $\CFK_{\F[U,V]}((4_1)_{2,1})$, which settles the case $n=2$. In all other cases, we will investigate what these two rectangles look like after taking a further $(2,1)$-cable.

Let us first consider the immersed curve corresponding to the horizontal $\frac{n}{2} \times 1$ rectangle and determine what happens to it after we take the $(2,1)$-cable. We start by drawing two copies of the curve, one slightly below and to the right of the other one. Each of them uses $\frac{n}{2}+1$ punctures. Focusing on the left copy, we observe that after sliding the punctures into a vertical line, all right arcs and all but the topmost and bottommost left arcs have length one. The topmost and bottommost left arcs have length two. See Figure \ref{fig:2,1 cable}. In terms of chain complexes, this immersed curve corresponds to the width $2$ staircase of shape $1, -1, \dots, 1, -1, 1, -2, -1, 1, \dots, -1, 1, -1, 2, \dots$, where each of the blocks of $1$'s and $-1$'s contains $n-1$ numbers, because there are $2(\frac{n}{2}+1)=n+2$ punctures in total. This width $2$ staircase is precisely the local system $K_n$ as required.
\begin{figure}[t]
\begin{subfigure}[b]{0.40\textwidth}
    \centering
    \begin{tikzpicture}[scale=0.9]
    \draw[very thick] (0,-6) -- (0,5);
    \draw[very thick] (3,-6) -- (3,5);
    \foreach \i in {0,1}{
    \begin{scope}[xshift = \i cm, yshift = -\i cm]
        \filldraw[] (1,4) circle (2pt);
        \filldraw[] (1,2) circle (2pt);
        \filldraw[] (1,0) circle (2pt);
        \filldraw[] (1,-2) circle (2pt);
        \filldraw[] (1,-4) circle (2pt);

        \draw[very thick] (1,4.5) to[out=180, in=120, looseness=0.8] (1, 3);
        \draw[very thick] (1,3) to[out=-60, in=0, looseness=0.4] (1, -4.5);
        \draw[very thick] (1,-3) to[out=60, in=0, looseness=0.4] (1, 4.5);
        \draw[very thick] (1,-3) to[out=240, in=180, looseness=0.9] (1, -4.5); 
    \end{scope}
    }
    \end{tikzpicture}
    \caption{}
    \label{subfigure:2 copies of horizontal rectangles}
    \end{subfigure}    
    \begin{subfigure}[b]{0.40\textwidth}
    \centering
    \begin{tikzpicture}[scale=0.9]
    \def\b{0.33}

    \foreach \i in {-3,-1, 1}{
        \begin{scope}[yshift = \i cm]
            \draw[very thick] (-0.5, 0) to[out = 90, in  = -90] (0.75, 1){};
            \draw[very thick] (-0.75, 0) to[out = 90, in  = -90] (0.5, 1){};
            \filldraw[] (0,0) circle (2pt);
        \end{scope}
    }
    \foreach \i in {-2,2}{
        \begin{scope}[yshift = \i cm]
            \draw[very thick] (0.5, 0) to[out = 90, in  = -90] (-0.75, 1){};
            \draw[very thick] (0.75, 0) to[out = 90, in  = -90] (-0.5, 1){};
            \filldraw[] (0,0) circle (2pt);
        \end{scope}
    }

    \draw[very thick] (0.5, 0) to[out = 90, in  = -90] (-0.50, 1){};
    \draw[very thick] (0.75, 0) to[out = 90, in  = -90] (-0.75, 1){};

    \filldraw[] (0, 0) circle (2pt);
    \filldraw[] (0, -4) circle (2pt);
    \filldraw[] (0, -5) circle (2pt);
    \filldraw[] (0, 3) circle (2pt);
    \filldraw[] (0, 4) circle (2pt);

    \draw[very thick] (-0.75, 3) to[in = 180, out =90] (0, 4.5);
    \draw[very thick] (-0.5, 3) to[in = 180, out =90] (0, 3.5);
    \draw[very thick] (0, 3.5) to[in =0, out =0] (0, 4.5);
    
    \draw[very thick] (-0.75, -3) to[in = 180, out =-90] (0, -4.5);
    \draw[very thick] (-0.5, -3) to[in = 180, out =-90] (0, -3.5);
    \draw[very thick] (0, -3.5) to[in =0, out =0] (0, -4.5);

    \draw[very thick] (-0.75, 0) ;
    
    \draw[very thick] (-1,-6) -- (-1,5);
    \draw[very thick] (1,-6) -- (1,5);
    \end{tikzpicture}
    \caption{}
    \label{subfigure:width 2 staircase exhibited}
    \end{subfigure}    
    \caption{Continuation of the calculation from Figure \ref{fig:cabling} -- a further $(2,1)$-cable is taken. Figure (\textsc{a}) shows two stretched copies of the immersed curve corresponding to the horizontal $\frac{n}{2}\times 1$ rectangle. Figure (\textsc{b}) shows what happens to the left immersed curve after translating the punctures so that they lie on the same vertical line. In terms of local systems, this immersed curve corresponds to $K_n$.}
    \label{fig:2,1 cable}
\end{figure}
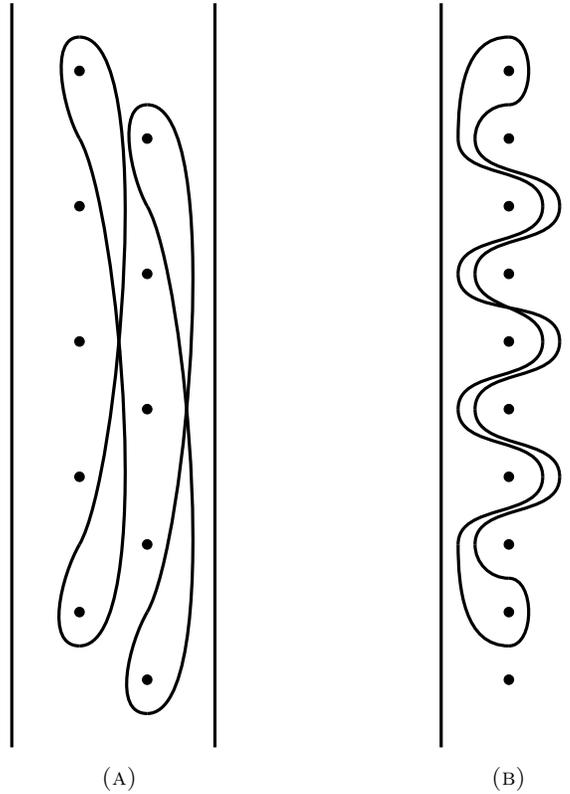

Finally, if $n \geq 3$ is odd, we consider the $(n, -2)$ cable of the figure eight knot. After drawing $n$ copies of the immersed curve in the first part of the cabling process, we focus on the middle copy $\gamma$ of the curve. Let $p_1$ and $p_{n+1}$ be the punctures that $\gamma$ is looping around and let $p_2, \dots, p_n$ be the punctures whose height lies between them, ordered so that higher punctures have a lower index. After sliding the punctures into a vertical line, the immersed curve will pass entirely to the left of $p_2$, entirely to the right of $p_3$, entirely to the left of $p_4$, and so on. In other words, except for the topmost left arc and the bottommost right arc, that have length $2$, the immersed curve of the $(n,-2)$-cable will keep alternating between the sides with arcs of length $1$. The local system that corresponds to this immersed curve has shape $1, -1, \dots, 1, -1, -2, 1, -1, \dots, 1, -1, 2, \dots$, where the first block of $1$'s and $-1$'s has $n+1$ terms and the second block of $1$'s and $-1$'s has $n-3$ terms. This is precisely the width $2$ staircase $K_n$ as required. 
\end{proof}
\begin{remark}
Proposition \ref{prop:realization} exhibits the local systems $K_n$ as direct summands of $\CFK_{\F[U,V]}(K)$ for some knots $K$. Correspondingly, the local systems $\K{n}$ are direct summands of the knot Floer complexes of their mirrors.
\end{remark}
Note that the information from immersed curves is in general only sufficient to determine $\CFK_{\mathcal{R}}(K)$ rather than $\CFK_{\F[U,V]}(K)$. However, this is not a problem in our case, since the complexes that correspond to the immersed curves from the proof admit a unique lift to chain complexes over $\F[U,V]$. These lifts are exactly $K_n$. 

Finally, not much was special about our use of the figure eight knot in the statement of Proposition \ref{prop:realization} -- the same local systems appear in cables of any knot whose $\CFK_{\F[U,V]}(K)$ contains a square. These are for example all nontrivial alternating knots and Whitehead doubles.

\bibliography{mybib.bib}
\bibliographystyle{alpha}

\end{document}